\documentclass[12pt]{amsart}
\usepackage{amsfonts,amssymb,latexsym,amsmath,amsthm,color,dsfont}
\usepackage{enumerate}
\usepackage{fullpage}

\usepackage{graphicx, amsmath, amssymb, amsthm,anysize,float,booktabs,geometry}
\usepackage{pdflscape}
\usepackage{caption} % Note that \captionof stands for 'caption outside float'
\usepackage{hyperref}

\hypersetup{citecolor=red, linkcolor=blue, colorlinks=true}

\marginsize{2cm}{2cm}{1cm}{3cm}
\allowdisplaybreaks

\newcommand{\Q}{\mathcal{Q}}

\newcommand{\M}{\mathcal{M}}
\newcommand{\Tr}{\mathrm{Tr}}
\newcommand{\PG}{\mathrm{PG}}
\newcommand{\N}{\mathrm{N}}

\theoremstyle{plain}
\newtheorem{theorem}{Theorem}[section]

\newtheorem{lemma}[theorem]{Lemma}

\newtheorem{proposition}[theorem]{Proposition}

\newtheorem{result}[theorem]{Result}
\numberwithin{equation}{section}
\theoremstyle{remark}

\newcommand\F{\mathbb{F}}
\newcommand\Z{\mathbb{Z}}

\newcommand\cl{\mathcal{L}}
\newcommand\cs{\mathcal{S}}
\newcommand\QQ{{\mathcal{Q}}}
\newcommand\PGO{\mathrm{PGO}}
\def\<{\langle}
\def\>{\rangle}

\begin{document}

\title[Cameron-Liebler Line Classes]{Cameron-Liebler Line Classes with parameter $x=\frac{(q+1)^2}{3}$}

\author[Feng, Momihara, Rodgers, Xiang, and Zou]{Tao Feng$^*$, Koji Momihara$^{\dagger}$, Morgan Rodgers, Qing Xiang, and Hanlin Zou}

\thanks{$^*$Research partially supported by the National Natural Science Foundation of China grant 11771392}
\thanks{$^{\dagger}$Research partially supported by  
JSPS under Grant-in-Aid for Young Scientists (B) 17K14236 and Scientific Research (B) 15H03636}
%\thanks{$^{\ddagger}$Research partially supported by an NSF grant DMS-1855723}

\address{Tao Feng, School of Mathematical Sciences, Zhejiang University, 38 Zheda Road, Hangzhou 310027, Zhejiang, P. R. China}
\email{tfeng@zju.edu.cn}

\address{Koji Momihara, Faculty of Education,Division of Natural Science\\
Faculty of Advanced Science and Technology, Kumamoto University, 2-40-1 Kurokami, Kumamoto 860-8555, Japan}
\email{momihara@educ.kumamoto-u.ac.jp}

\address{Morgan Rodgers, Department of Mathematics, Fresno State University, Fresno, CA 93740 USA} \email{morgan@csufresno.edu}

\address{Qing Xiang, Department of Mathematical Sciences, University of Delaware, Newark, DE 19716, USA} \email{qxiang@udel.edu}

\address{Hanlin Zou,  Department of Mathematical Sciences, University of Delaware, Newark, DE 19716, USA} \email{hanlin@udel.edu}

\keywords{Cameron-Liebler line class, Gauss sum, Klein quadric,  %partial difference set,
 spread, tight set}

\begin{abstract}

 Cameron-Liebler line classes were introduced in \cite{CL}, and motivated by a question about orbits of collineation groups of $\PG(3,q)$. These line classes have appeared in different contexts under disguised names such as Boolean degree one functions, regular codes of covering radius one, and tight sets.  In this paper we construct an infinite family of Cameron-Liebler line classes in $\PG(3,q)$ with new parameter $x=(q+1)^2/3$ for all prime powers $q$ congruent to 2 modulo 3. The examples obtained when $q$ is an odd power of two represent the first infinite family of Cameron-Liebler line classes in $\PG(3,q)$, $q$ even.

\end{abstract}

%%%%%%%%%%%%%%%%%%%%%%%%%%%%%%%%%%%%%%%%%%%%%%

\maketitle

\section{Introduction}

Let $q$ be a prime power and let $\PG(3,q)$ be the 3-dimensional projective space over the finite field $\F_q$ of order $q$. A {\it spread} in $\PG(3,q)$ is a set of its lines which partitions its points. Let $\cl$ be a set of lines of $\PG(3,q)$ with $|\cl|=x(q^2+q+1)$, $x$ a positive integer. We say that $\cl$ is a {\it Cameron-Liebler line class with parameter $x$} if $|\cl \cap \cs|=x$ for all spreads $\cs$ of $\PG(3,q)$. For example, let $\cl$ be either the set of all lines passing through a fixed point $P$ of $\PG(3,q)$ or the set of all lines in a plane $\pi$ of $\PG(3,q)$; then we see that $\cl$ is a Cameron-Liebler line class with parameter $1$. Furthermore the union of these two sets for $P\not\in \pi$ forms a Cameron-Liebler line class with parameter $x=2$. These examples of Cameron-Liebler line classes with parameter $x=1$ or $2$ are called {\it trivial}. Also, the complement of a Cameron-Liebler line class with parameter $x$ in the set of all lines of $\PG(3,q)$ is a Cameron-Liebler line class with parameter $q^2+1-x$. So without loss of generality we may assume that $ x\leq (q^2+1)/2$ when discussing Cameron-Liebler line classes of parameter $x$ in $\PG(3,q)$.

Cameron-Liebler line classes were first introduced by Cameron and Liebler \cite{CL} in their study of collineation groups of $\PG(3,q)$ having the same number of orbits on points and lines of $\PG(3,q)$. Penttila \cite{tp1, tp2} coined the term ``Cameron-Liebler line class" and studied these objects in some depth. Bruen and Drudge \cite{bd} constructed the first infinite family of Cameron-Liebler line classes with parameter $x=(q^2+1)/2$ for all odd prime powers $q$. After much study of Cameron-Liebler line classes in $\PG(3,q)$,  the notion of  Cameron-Liebler line classes has been generalized to Cameron-Liebler $k$-classes \cite{rsv} in $\PG(2k+1,q)$, and to Cameron-Liebler sets of generators in finite classical polar spaces \cite{drss}. In fact, Cameron-Liebler sets can be introduced for any distance-regular graph; this was done previously under various names: Boolean degree one functions, completely regular codes of strength 0 and covering radius 1, and tight sets. We refer the reader to \cite{fi} for more details on these connections. In this paper we will focus on Cameron-Liebler line classes in $\PG(3,q)$.

The central problem concerning Cameron-Liebler line classes in $\PG(3,q)$ is: for which values of the parameter $x$, $1\leq x\leq (q^2+1)/2$, do there exist Cameron-Liebler line classes with parameter $x$? On the nonexistence side, the state-of-the-art results are those in \cite{metsch2014} and \cite{gavmetsch}.
In particular, it is shown in \cite{metsch2014} that there are no Cameron-Liebler line classes with parameter $x$ in $\PG(3,q)$ if $3\leq x\leq q\sqrt[3]{q/2}-2q/3$.
In terms of constructive results, infinite families of Cameron-Liebler line classes with parameter $x=(q^2+1)/2$ and $x=(q^2-1)/2$ have been constructed in \cite{bd, cp1, cp2, ddmr, fmx, gmp}.
Even though there have been a large number of papers on constructing Cameron-Liebler line classes in $\PG(3,q)$, the parameters of the known infinite families are restricted to either $(q^2+1)/2$ or $(q^2-1)/2$.
In particular, no infinite families of nontrivial Cameron-Liebler line classes in $\PG(3,q)$ are known when $q$ is a power of 2 (note that there are a few examples of Cameron-Liebler line classes known in $\PG(3,4)$, $\PG(3,8)$, $\PG(3,32)$, and $\PG(128)$; see \cite{goP, rodthesis}).
In this paper, we construct Cameron-Liebler line classes in $\PG(3,q)$ with parameter $x=(q+1)^2/3$ for all $q$ congruent to 2 modulo 3. In particular, the first infinite family of Cameron-Liebler line classes in $\PG(3,q)$, $q$ an odd power of 2, is constructed here.

We give an overview of our construction here.
The initial step is to prescribe an automorphism group for the Cameron-Liebler line classes that we intend to construct; once this is done, the Cameron-Liebler line classes we want to construct will be unions of orbits of lines under the action of the prescribed automorphism group.
For the choices of automorphism groups, we follow the idea in \cite{rodthesis}; that is, we will choose a cyclic group of order $q^2+q+1$ as the prescribed automorphism group. Examples of Cameron-Liebler classes with parameter $(q+1)^2/3$ have been found in this way by using a computer for all $q < 150$ with $q \equiv 2\bmod{3}$ (see \cite{rodthesis}).
The difficulty lies in how to come up a choice of orbits for general $q$ %(which is a prime power congruent to 2 modulo 3) so that the union of chosen orbits
which will always give a Cameron-Liebler line class in $\PG(3,q)$ with parameter $(q+1)^2/3$. The examples in \cite{rodthesis} provided vital clues for a general choice; also the computations of additive character sums (needed to prove that the union of the chosen orbits is a Cameron-Liebler line class) gave us hints for making correct choices of orbits. In Section 3, we come up with an explicit choice of orbits that will give a Cameron-Liebler line classes with parameter $x=(q+1)^2/3$ for all $q$ congruent to $2$ modulo $3$.

The paper is organized as follows. In Section~\ref{sec:pre}, we review background material on Cameron-Liebler line classes, $x$-tight sets in $\Q^+(5,q)$, %strongly regular Cayley graphs,
and character sums over finite fields. In Section~\ref{sec:const0}, we %work  and
introduce two multisets $D_1$ and $D_2$ of $\F_{q^3}^*$ which will be crucial for choosing orbits. In Section~\ref{sec:const2}, we give the proofs that our choice of orbits will indeed give Cameron-Liebler line classes; since the prescribed group is a cyclic one, the computations of additive character sums necessarily involve Gauss sums. In Section 5, we determine the stabilizers of our Cameron-Liebler line classes in $\textup{PSL}(4,q)$. In the Appendix, we give some computations of exponential sums needed in the proof of our main theorem.

\section{Preliminaries}\label{sec:pre}

\subsection{Cameron-Liebler line classes in $\PG(3,q)$ and tight sets in $\Q^+(5,q)$}

To investigate Cameron-Liebler line classes in $\PG(3,q)$, it is often useful to translate their definition to the setting of $\Q^+(5,q)$ using the Klein correspondence (here $\Q^+(5,q)$ is the 5-dimensional hyperbolic orthogonal space, also known as the {\it Klein quadric}). Let $x$ be a positive integer. A subset $\mathcal{M}$ of the points of $\Q^+(5,q)$ is called an {\it $x$-tight set} if for every point $P\in \Q^+(5,q)$, $|P^\perp\cap\mathcal{M}|=x(q+1)+q^2$ or $x(q+1)$ according as $P$ is in $\mathcal{M}$ or not, where $\perp$ is the polarity determined by $\Q^+(5,q)$. The geometries of $\PG(3,q)$ and $\Q^+(5,q)$ are closely related through a mapping known as the Klein correspondence which maps the lines of $\PG(3,q)$ bijectively to the points of $\Q^+(5,q)$, c.f. \cite{h2,stan}. Let $\cl$ be a set of lines of $\PG(3,q)$ with $|\cl|=x(q^2+q+1)$, $x$ a positive integer, and let $\M$ be the image of $\cl$ under the Klein correspondence. Then it is known that $\cl$ is a Cameron-Liebler line class with parameter $x$ in $\PG(3,q)$ if and only if $\M$ is an $x$-tight set of  $\Q^+(5,q)$. Moreover, if $\cl$ is a Cameron-Liebler line class with parameter $x$, by \cite[Theorem 2.1(b)]{metsch2014} it holds that $|P^\perp\cap \M|=x(q+1)+q^2$ for any point $P\in\M$ and $|P^\perp\cap \M|=x(q+1)$ for any point $P\notin \M$ (here $P$ can be in the exterior of $\Q^+(5,q)$); consequently $\M$ is a projective two-intersection set in $\PG(5,q)$ with intersection sizes $h_1=x(q+1)+q^2$ and $h_2=x(q+1)$. We summarize these known facts as follows.

\begin{result}\label{res1}
 Let $\cl$ be a set of $x(q^2+q+1)$ lines in $\PG(3,q)$ with $1\leq x\leq (q^2+1)/2$, and let $\M$ be the image of $\cl$ under the Klein correspondence. Then $\cl$ is a Cameron-Liebler line class with parameter $x$ if and only if $\M$ is an $x$-tight set in $\Q^+(5,q)$; moreover, in the case when $\cl$ is a Cameron-Liebler line class, we have
 \[|P^\perp\cap \M|=\begin{cases}
   x(q+1)+q^2, & \emph{if } P\in \M, \\
   x(q+1),     & \emph{otherwise}.\end{cases}\]
\end{result}

Let $\cl$ be a Cameron-Liebler line class with parameter $x$ in $\PG(3,q)$ and let $\M\subset \Q^+(5,q)$ be the image of $\cl$ under the Klein correspondence. By Result \ref{res1}, $\M$ is a projective two-intersection set in $\PG(5,q)$. %By \cite{CK1986}, we can construct a corresponding strongly regular Cayley graph as follows.
Define $D=\{\lambda v: \lambda\in\F_{q}^*, \<v\>\in\M\}$, which is a subset of $(\F_{q}^6,+)$. %Then the Cayley graph with vertex set $(\F_q^6,+)$ and connection set $D$ is strongly regular. Its restricted eigenvalues can be determined as follows.
Let $\psi$ be a non-principal additive character of $\F_q^6$. Then $\psi$ is principal on a unique hyperplane $P^\perp$ for some point $P\in\PG(5,q)$. We have
\begin{align*}
 \psi(D) & =\sum_{\<v\>\in\M}\sum_{\lambda\in\F_q^*}\psi(\lambda v)=\sum_{\<v\>\in\M}(q\mathds{1}_{P^\perp}(\<v\>)-1) \\
         & =-|\M|+q|P^\perp\cap \M|=\begin{cases}
  -x+q^3, & \text{if~}P\in\M, \\
  -x,     & \text{otherwise},\end{cases}
\end{align*}
where $\mathds{1}_{P^\perp}(\<v\>)$ is the characteristic function taking value 1 if $\<v\>\in P^\perp$, and 0 otherwise. Conversely, for each point $P\in\PG(5,q)$, there is a non-principal character $\psi$ that is principal on the hyperplane $P^\perp$, and the size of $P^\perp\cap\M$ can be computed from $\psi(D)$. Therefore the character values of $D$ reflect the sizes of intersection of $\M$ with the hyperplanes of $\PG(5,q)$. To summarize, we have the following result.

\begin{result}\label{res2}
 Let $\cl$ be a set of $x(q^2+q+1)$ lines in ${\rm PG}(3,q)$ with $1\leq x \leq (q^2+1)/2$, and let $\M$ be the image of $\cl$ under the Klein correspondence. Define
 $$D=\{\lambda v: \lambda\in\F_q^*, \<v\>\in\M\}\subset (\F_q^6,+).$$
 Then $\cl$ is a Cameron-Liebler line class with parameter $x$ if and only if $|D|=(q^3-1)x$ and for any $P\in {\rm PG}(5,q)$,
 \[\psi(D)=\begin{cases}
   -x+q^3, & \emph{if }P\in\M, \\
   -x,     & \emph{otherwise},
  \end{cases}\]
 where $\psi$ is any non-principal character of $\F_q^6$ that is principal on the hyperplane $P^\perp$.
\end{result}

%The essence of Result~\ref{res2} is that a Cameron-Liebler line class with parameter $x$ in $\PG(3,q)$ is the same object as a ``self-dual" partial difference set in $(V,+)$ with appropriate parameters.

\subsection{Gauss sums}
We collect some auxiliary results on Gauss sums as a preparation for computing additive character values of a subset of vectors of a vector space over $\F_q$. We assume that the reader is familiar with the basic theory of characters of finite fields as can be found in Chapter 5 of \cite{ln1997}

Let $q=p^n$ with $p$ a prime and $n \ge 1$, and let $\zeta_p=\exp(\frac{2\pi \sqrt{-1}}{p})$. Furthermore, let $\psi_{\F_{q}}$ be the {\it canonical} additive character of $\F_q$ defined by $\psi_{\F_{q}}(x)=\zeta_p^{\Tr_{q/p}(x)}$, where $\Tr_{q/p}$ is the absolute trace from $\F_q$.
 For any multiplicative character $\chi$ of $\F_q$, %and an additive character $\psi$ of $\F_q$,
 define the {\it Gauss sum} by
 $$G_q(\chi)=\sum_{x\in\F_{q}^*}\psi_{\F_q}(x)\chi(x).$$
%When the additive character $\psi$ is $\psi_{\F_q}$ (the canonical additive character of $\F_q$), we simply write the Gauss sum $G_q(\psi, \chi)$ as $G_q(\chi)$.
The following are some basic properties of Gauss sums:
\begin{enumerate}[(i)]
 \item $G_q(\chi)\overline{G_q(\chi)}=q$ if $\chi$ is non-principal;
 \item $G_q(\chi^{-1})=\chi(-1)\overline{G_q(\chi)}$;
 \item $G_q(\chi)=-1$ if $\chi$ is principal.
\end{enumerate}

Gauss sums are instrumental in the transition from the additive to the multiplicative structure (or the other way around) of a finite field. This can be seen more precisely in the next lemma.

\begin{lemma}\label{orthchar}
 By orthogonality of
 characters, the canonical additive character $\psi_{\F_q}$ of $\F_q$ can be expressed as a linear combination of the multiplicative characters:
\begin{equation}\label{eqn_GSexp}
  \psi_{\F_q}(x)=\frac{1}{q-1}\sum_{\chi\in \widehat{\F_q^*}}G_q(\chi^{-1})\chi(x), \ \forall x\in\F_q^*,
 \end{equation}
 where $\widehat{\F_q^*}$ is the character group of $\F_q^\ast$.
 On the other hand, each nontrivial multiplicative character $\chi$ of $\F_q$ can also be expressed as a linear combination of the additive characters: 
 \[
  \chi(x)=\frac{1}{q}G_q(\chi)\sum_{a\in \F_q^\ast}\chi^{-1}(-a)\psi_{\F_q}(ax), \ \forall x\in\F_q^*.\]
\end{lemma}

\begin{lemma}\label{partialGS}
 Let $C_0$ be a subgroup of $\F_q^*$ of index $N$, and let $\chi$ be a character of $\F_q^*$ of order $N$. Then for any $x\in\F_q^*$ we have
 \[
  \frac{1}{N}\sum_{j=0}^{N-1}G_q(\chi^{-j})\chi^j(x)=\sum_{a\in C_0}\psi_{\F_q}(xa).
 \]
\end{lemma}
\begin{proof}
 Let $\theta$ be a character of $\F_q^*$ of order $q-1$, and let $\chi=\theta^{(q-1)/N}$. By \eqref{eqn_GSexp}, we have
 \begin{equation}\label{short}
\sum_{a\in C_0}\psi_{\F_q}(xa)=\frac{1}{q-1}\sum_{i=0}^{q-1}G_q(\theta^{-i})\theta^i(x)\sum_{a\in C_0}\theta^i(a).
 \end{equation}
 The inner sum in the right hand side of (\ref{short}) equals $(q-1)/N$ when $i\equiv 0\pmod{(q-1)/N}$, and $0$ otherwise; so $\sum_{a\in C_0}\psi_{\F_q}(xa)$ equals $\frac{1}{N}\sum_{j=0}^{N-1}G_q(\chi^{-j})\chi^j(x)$ as desired.
\end{proof}

\bigskip
The following result on the character values of a Singer difference set will be used in the proof of our main theorem.
\begin{lemma}[{\cite[Theorem 2.1]{EHKX1999}}]\label{chiL0}
 Let $L$ be a complete set of coset representatives of $\F_q^*$ in $\F_{q^3}^*$. Let $${\mathcal S}=\{x\in L\mid {\rm Tr}_{q^3/q}(x)=0\}.$$ If $\chi$ is a nontrivial character of $\F_{q^3}^*$ whose restriction on $\F_q^*$ is trivial, then
 $$\chi({\mathcal S})=G_{q^3}(\chi)/q.$$
\end{lemma}

\subsection{Cubic polynomials over $\F_q$}
Let $q=p^n$ be a prime power, where $p\ne 3$ is a prime. Let $f(X)=X^3+cX+d$ be a cubic polynomial  over $\F_q$, and let $\gamma_1,\gamma_2,\gamma_3$ be its roots in some extension field of $\F_q$. The discriminant of $f$ is
\[
 \Delta(f):=(\gamma_1-\gamma_2)^2(\gamma_2-\gamma_3)^2(\gamma_3-\gamma_1)^2,
\]
which equals $-4c^3-27d^2$ for all $q$. Hence $f$ has no repeated roots if and only if $\Delta(f)\ne 0$. In particular, when $q$ is even, we have $\Delta(f)=d^2$. We shall need the following theorem giving the number of roots of $f$ in $\F_q$ in various situations.
\begin{theorem}\cite{Dixon,Williams1975}\label{cubic}
 Let $p\ne 3$ be a prime and $q=p^n$. Suppose that  $f(X)=X^3+cX+d$ is a polynomial over $\F_q$ with discriminant $\Delta(f)\ne 0$.
 \begin{enumerate}
  \item[(i)] If $q$ is odd, $f$ has exactly one root in $\F_q$ if $\Delta(f)$ is a nonsquare in $\F_q$ and $0$ or $3$ roots in $\F_q$ otherwise.
  \item[(ii)] If $q$ is even, $f$ has exactly one root in $\F_q$ if $\Tr_{q/2}(c^3d^{-2})\ne \Tr_{q/2}(1)$ and $0$ or $3$ roots in $\F_q$ otherwise.
 \end{enumerate}
\end{theorem}

Our construction of new Cameron-Liebler line classes is based on the image sets of certain cubic polynomials as shown in the next section. This idea was previously used in \cite{DD2004} for constructing new difference sets with Singer parameters.

\section{Cameron-Liebler line classes with parameter $x=(q+1)^2/3$}\label{sec:const0}

\subsection{The set $E$}\label{sec:const1}
Throughout the rest of the paper, we always assume that $q$ is a prime power such that $q\equiv 2\pmod{3}$. We define
\begin{align*}
 T_0 & =\{x\in\F_{q^3}^*:\, \Tr_{q^3/q}(x)=0\}, \\
 L_0 & =\{x\in T_0: \,\N_{q^3/q}(x)=1 \},
\end{align*}
where $\Tr_{q^3/q}$ and $\N_{q^3/q}$ are the relative trace and norm from $\F_{q^3}$ to $\F_q$, respectively. Then $|T_0|=q^2-1$, $|L_0|=q+1$ and $L_0\cdot \F_q^*=T_0$. Since $\gcd(q-1,q^2+q+1)=1$, we have $C_0\cdot \F_q^*=\F_{q^3}^*$, where $C_0$ is the subgroup of $\F_{q^3}^*$ of order $q^2+q+1$.

\begin{lemma}\label{neg3ns}
   If $q\equiv 2\pmod{3}$ with $q$ odd, then $-3$ is a nonsquare in $\F_q$.
  \end{lemma}
  \begin{proof}
   Write $q=p^n$ with $p$ an odd prime. Then $p\equiv 2\pmod{3}$ and $n$ is odd. It suffices to show that $-3$ is a nonsquare in $\F_p$. By the quadratic reciprocity we have
   \[\left(\frac{-3}{p}\right)=\left(\frac{-1}{p}\right)\cdot \left(\frac{3}{p}\right)=(-1)^{(p-1)/2}\cdot(-1)^{(p-1)/2\cdot (3-1)/2}\cdot\left(\frac{p}{3}\right)=-1.\]
   Here, $(\frac{\cdot}{p})$ is the Legendre symbol. The proof is complete.
  \end{proof}

\begin{lemma}\label{lem_Qz}
 If $z$ is an element of $\F_{q^3}^*$ such that $\Tr_{q^3/q}(z)=0$, then  $\Tr_{q^3/q}(z^{1+q})\ne 0$.
\end{lemma}
\begin{proof}
 We have $z\not\in\F_q$, since otherwise $3z=\Tr_{q^3/q}(z)=0$. If $\Tr_{q^3/q}(z^{1+q})=0$, then the minimal polynomial of $z$ over $\F_q$ is $X^3-c$, where $c=\N_{q^3/q}(z)$. Since $q\equiv 2\pmod{3}$, we have $\gcd(q-1,3)=1$, so $X^3=c$ has exactly one root in $\F_q$: a contradiction to the irreducibility of $X^3-c$. This completes the proof.
\end{proof}

Since $q\equiv 2\pmod{3}$, we have $\gcd(q-1,3)=1$ and so the map $y\mapsto y^3$ is a permutation of $\F_q$. We write $y\mapsto y^{1/3}$ for its inverse map. We define two multisets as follows:
\begin{align*}
 D_1 & =[x\N_{q^3/q}(\lambda+x^q-x^{q^2})^{1/3}: x\in L_0,\lambda\in \F_q],             \\
 D_2 & =[\beta^{-1} x\N_{q^3/q}(\lambda+x^q-x^{q^2})^{-1/3}: x\in L_0,\lambda\in \F_q],
\end{align*}
where $\beta=-3^{-1}\in\F_q$. Set $\gamma:=\beta^{-3}=-27$.

\begin{lemma}\label{calpha}
 Let $x\in L_0$ and set $z:=x^q-x^{q^2}$. For each $\alpha\in\F_q^*$, set
 \[
  c_\alpha:=|\{\lambda\in \F_q:\, \N_{q^3/q}(\lambda+z)=\alpha\}|+|\{\lambda\in \F_q:\, \gamma \N_{q^3/q}(\lambda+z)^{-1}=\alpha\}|.
 \]
 Then $c_\alpha=1$ or $4$.
\end{lemma}
\begin{proof}
 Write $a:=\Tr_{q^3/q}(z^{1+q})$, $b:=\N_{q^3/q}(z)$, and set $u:=-\frac{1}{3}a$. It is clear that $\Tr_{q^3/q}(z)=0$, so  $a\ne 0$ by Lemma \ref{lem_Qz}. The minimal polynomial of $x$ over $\F_q$ is $g(X):=X^3+eX-1$ with $e=\Tr_{q^3/q}(x^{1+q})$. The discriminant of $g$ is  $z^{2+2q+2q^2}=b^2$ by definition, and it equals $-4e^3-27$. So $b^2=-4e^3+\gamma$. On the other hand, we have
 \begin{align*}
  a & =\Tr_{q^3/q}((x^q-x^{q^2})^{1+q})=\Tr_{q^3/q}(x^{q+1}-x^2) \\
    & =\Tr_{q^3/q}(x^{q+1}+x(x^q+x^{q^2}))=3e.
 \end{align*}
 We thus have $b^2-4u^3=\gamma$, which is a nonsquare in $\F_q$ when $q$ is odd by Lemma \ref{neg3ns}.

 For any element $\alpha\in\F_q^*$, let $f_1(X):=X^3+aX+b-\alpha$, $f_2(X):=X^3+aX+b-\beta^{-3}\alpha^{-1}$. Upon expansion we deduce that $$c_\alpha=|\{\lambda\in\F_q:\, f_1(\lambda)=0\}|+|\{\lambda\in\F_q:\, f_2(\lambda)=0\}|.$$
 The discriminants of the two cubic polynomials $f_1$ and $f_2$ are
 \begin{align*}
  \Delta_1 & =-4a^3-27(b-\alpha)^2=\gamma((\alpha-b)^2-4u^3), \\ \Delta_2&=-4a^3-27(b-\beta^{-3}\alpha^{-1})^2=\gamma((\gamma \alpha^{-1}-b)^2-4u^3),
 \end{align*}
 respectively. We claim that either none or both of $\Delta_1,\Delta_2$ are zero, and if the latter occurs then $\alpha^2-2b\alpha+\gamma=0$. We compute that
 \begin{align*}
  \Delta_1\Delta_2= &
  \gamma^2((\alpha-b)^2-4u^3)((\gamma \alpha^{-1}-b)^2-4u^3)                                                                 \\
  =                 & \gamma^2(\alpha^2-2\alpha b+\gamma)(\gamma^2\alpha^{-2}-2\gamma\alpha^{-1} b+\gamma)                   \\
  =                 & \gamma^2(\alpha^2\gamma+\gamma^3\alpha^{-2}-4b\gamma\alpha-4b\gamma^2\alpha^{-1}+2\gamma^2+4b^2\gamma) \\
  =                 & \gamma^3(\alpha+\gamma\alpha^{-1})^2+4b\gamma^3(b-\alpha-\gamma\alpha^{-1})                            \\
  =                 & \gamma^3(\alpha+\gamma\alpha^{-1}-2b)^2.
 \end{align*}
 If $\Delta_1=0$, then $\Delta_1\Delta_2=0$ and so $\alpha-b=b-\gamma\alpha^{-1}$, which in turn implies that $\Delta_2=0$. The converse is also true. The claim is now established.

 We first consider the case where $\Delta_1=\Delta_2=0$, so that $\alpha^2-2b\alpha+\gamma=0$. In this case, $f_1(X)$ has a repeated root $\eta$. If $\eta$ is not in $\F_q$, then $\eta^q$ would also be a repeated root of $f_1(X)$, contradicting the fact that $\deg(f_1)=3$. Hence all the roots of $f_1(x)=0$ lies in $\F_q$. The three roots of $f_1$ can not be all equal: if $f_1(X)=(X-\eta)^3$, then $\eta=0$ by comparing the coefficients of $X^2$, and so $a=0$: a contradiction. To conclude, $f_1(X)$ has two distinct roots in $\F_q$. The same is true for $f_2$ by the same argument. We thus deduce that $c_\alpha=4$ in this case.

 We next consider the case $\Delta_1\Delta_2\ne 0$.  In the case where $q$ is odd, $\Delta_1\Delta_2$ is a nonsquare in $\F_q$ by Lemma~\ref{neg3ns}. That is, exactly one of $\Delta_1$ and $\Delta_2$ is a nonsquare of $\F_q$.  By (i) of Theorem \ref{cubic}, we conclude that one of $f_1$ and $f_2$ has exactly one zero in $\F_q$ and the other has $0$ or $3$ zeros in $\F_q$. It follows that $c_\alpha=1$ or $4$ as desired. In the case where $q$ is even, we have $\gamma=\beta=1$ and  $z=x$. It follows that $a=\Tr_{q^3/q}(x^{1+q})=\Tr_{q^3/q}(x^{-q^2})$ and $b=\N_{q^3/q}(x)=1$ by the fact $x\in L_0$. The two cubic polynomials take the form $X^3+aX+1+\alpha$,  $X^3+aX+1+\alpha^{-1}$, respectively. Since $x\in L_0$, we have  $\Tr_{q^3/q}(x)=0$ and $x^{1+q+q^2}=1$. We compute that
 \begin{align*}
    & \Tr_{q/2}\left(\frac{a^3}{1+\alpha^2}\right)+\Tr_{q/2}\left(\frac{a^3}{1+\alpha^{-2}}\right)
  =\Tr_{q/2}(a^3)                                                                                  \\
  = & \Tr_{q/2}\Big((x^{-1}+x^{-q}+x^{-q^2})^3\Big)=\Tr_{q^3/2}(x^{-3}+x^{q-1}+x^{q^2-1})          \\
  = & \Tr_{q^3/2}(x^{q^2+q+1-3}+1)=\Tr_{q^3/2}(x^{q-1}\cdot x^{q^2-1})+1                           \\
  = & \Tr_{q^3/2}(x^{q-1}x^{-1}(x+x^q))+1= \Tr_{q^3/2}(x^{2(q-1)}+x^{q-1})+1                       \\
  = & 1.
 \end{align*}
 As in the case where $q$ is odd, we deduce that $c_\alpha=1$ or $4$ by using (ii) of Theorem \ref{cubic}.
\end{proof}

\begin{proposition}\label{keyodd}
 There is a subset $E\subseteq \F_{q^3}^*$ of size $\frac{(q+1)^2}{3}$ such that
 $$D_1+D_2=3E+T_0$$
 in the group ring $\mathbb{Z}[\F_{q^3}^*]$.
\end{proposition}
\begin{proof}
 For $x\in L_0$, set $z:=x^q-x^{q^2}$, and define a multiset
 \begin{equation}\label{eqn_Wx}
    W_x:=[\N_{q^3/q}(\lambda+z): \lambda\in\F_q]\cup[\gamma\N_{q^3/q}(\lambda+z)^{-1}: \lambda\in\F_q].
   \end{equation}
 For any element $\alpha\in\F_q^*$, its multiplicity in $W_x$ equals $c_\alpha$, where $c_\alpha$ is as defined in Lemma \ref{calpha}. We have $c_\alpha\in\{1,4\}$ by the same lemma.
 Therefore, there is a subset $L_x$ of $\F_q^*$ such that $W_x=\F_q^*+3L_x^{(3)}$, where $L_x^{(3)}=\sum_{z\in L_x}z^3$. Moreover,
 it is clear that $|L_x|=\frac{2q-|\F_q^*|}{3}=\frac{q+1}{3}$.

 It is routine to check that $D_1+D_2=\sum_{x\in L_0}xW_x^{(1/3)}$. Set $E:=\sum_{x\in L_0}xL_x$, which is a subset of $\F_{q^3}^*$ of size $(q+1)^2/3$. Then the claim in the proposition follows from the fact $W_x=\F_q^*+3L_x^{(3)}\in \mathbb{Z}[\F_{q^3}^*]$ for $x\in L_0$, and $L_0\cdot \F_q^*=T_0$. The proof is now complete.
\end{proof}

\subsection{The set $\M$}\label{subsec_M}
Let $V=\F_{q^3}\times \F_{q^3}$, which is viewed as a 6-dimensional vector space over $\F_{q}$. Define a map $Q: V\to\F_q$ by
\[Q((x,y))=\Tr_{q^3/q}(xy),\ \forall (x,y)\in V.\]
It is easy to check that $Q$ is a non-degenerate hyperbolic quadratic form on $V$. The quadric defined by $Q$ will be our model for $\Q^+(5,q)$ whose points can be expressed as $\{\<(x,y)\>: Q(x,y)=0\}$. The polar form $f: V\times V\rightarrow \F_q$ of $Q$ is given by
$$f((x,y),(a,b))=\Tr_{q^3/q}(bx+ay).$$
For a point $P=\<(x_0,y_0)\>$, its polar hyperplane $P^\perp$ is given by
\[
 P^\perp=\{\<(x,y)\>: \Tr_{q^3/q}(xy_0+x_0y)=0\}.
\]

Let $w$ be a primitive element of $\F_{q^3}$. Let $C_0$ be the subgroup of $\F_{q^3}^*$ of order $q^2+q+1$. For any $\mu\in C_0$, consider the following map on $\QQ^+(5,q)$ defined by
$$\langle (x,y)\rangle \mapsto \langle(\mu x, \mu^{-1}y)\rangle.$$
Then,  $C_0$ is embedded as a subgroup $i(C_0)$ of $\PGO^+(6,q)$, and it acts semi-regularly on the points of $\QQ^+(5,q)$. So the points of $\QQ^+(5,q)$ are partitioned into orbits of this action; each orbit has length $q^2+q+1$, and the number of orbits is $q^2+1$. We denote the orbit containing the point $\langle (a,b)\rangle$ by $O_{(a,b)}$. Then, all the orbits are $O_{(0,1)}$ and $O_{(1,z)}$, $z\in \F_{q^3}$ with $\Tr_{q^3/q}(z)=0$. The following is our main theorem.

\begin{theorem}\label{mtheven}
 Let $q\equiv 2\pmod{3}$ be a prime power.
 Let $\M=\bigcup_{z\in E}O_{(1,z)}$, where $E$ is defined in Proposition~\ref{keyodd}. Then, the line set ${\mathcal L}$ in $\PG(3,q)$ corresponding to $\M$ under the Klein correspondence forms a Cameron-Liebler line class with parameter $x=(q+1)^2/3$.
\end{theorem}

Let $\M$ be defined as in Theorem \ref{mtheven}. Clearly $|\M|=x(q^2+q+1)$ with $x=(q+1)^2/3$. Set $D=\{\lambda v: \lambda\in\F_q^*, \<v\>\in\M\}$. Then
$$|D|=(q-1)|\M|=(q-1)|E|(q^2+q+1)=(q^3-1)\frac{(q+1)^2}{3}.$$
To prove Theorem \ref{mtheven}, we need to show that $D$ has the correct character values as specified in Result~\ref{res2}. Each additive character of $(V,+)$ is of the form $\psi_{a,b}$ for some $(a,b)\in V$, where
$$\psi_{a,b}((x,y))=\psi_{\F_{q}}(f((a,b),(x,y)))=\psi_{\F_{q^3}}(bx+ay),\;\forall (x,y)\in V$$
Here $\psi_{\F_q}$ and $\psi_{\F_{q^3}}$ are the canonical additive characters of $\F_q$ and $\F_{q^3}$, respectively. It is clear that $\psi_{a,b}$ is trivial on the hyperplane  $\<(a,b)\>^\perp$.
By Result \ref{res2}, in order to prove Theorem \ref{mtheven} it suffices to prove the following claim: for any nonzero element $(a,b)$ of $V$, we have
\begin{equation}\label{toprove}
 \psi_{a,b}(D)=\begin{cases}
  -\frac{(q+1)^2}{3}+q^3, & \text{if~}(a,b)\in D, \\
  -\frac{(q+1)^2}{3},     & \text{otherwise}.\end{cases}
\end{equation}
This will be accomplished in the next section.

%%%%%%%%%%%%%%%%%%%%%%%%%%%%%%%%%%%%%%%%%%%%%%%%%%%%%%%%%%%
%%%%%%%%%%%%%%%%%%%%%%%%%%%%%%%%%%%%%%%%%%%%%%%%%%%%%%%%%%%

\section{Proof of the main theorem}\label{sec:const2}

Let $\M$ be defined as in Theorem \ref{mtheven}, and let $D$ be the corresponding subset of nonzero vectors in $V=\F_{q^3}\times\F_{q^3}$. Take the same notation as introduced in Subsection \ref{subsec_M}, and set $N:=q^2+q+1$. To simplify notation, we write $G(\chi)$ for the Gauss sum $G_{q^3}(\chi)$, where $\chi$ is a multiplicative character of $\F_{q^3}^*$. We need to evaluate the character sum $$\psi_{a,b}(D)=\sum_{z\in E}\psi_{a,b}(\F_{q}^*O_{(1,z)})$$
for nonzero $(a,b)\in V$, where $\F_q^\ast O_{(1,z)}=\{(y\mu,y\mu^{-1}z): y \in \F_q^\ast,\mu\in C_0\}$. For $z\in E$, we have
\begin{equation}\label{kloost}
 \psi_{a,b}(\F_{q}^*O_{(1,z)})=\sum_{y\in\F_q^*}\sum_{\mu\in C_0}\psi_{\F_{q^3}}(by\mu+ay\mu^{-1}z).
\end{equation}
The inner sum in the right hand side of (\ref{kloost}) is an incomplete Kloosterman sum. It is in general very difficult to evaluate incomplete Kloosterman sums exactly. Here we are dealing with certain sums $\psi_{a,b}(D)$ of incomplete Kloosterman sums; and for these sums we can evaluate them exactly.

\begin{lemma}\label{ab0}
 If $ab=0$ but $(a,b)\neq (0,0)$, then $\psi_{a,b}(D)=-\frac{(q+1)^2}{3}$.
\end{lemma}
\begin{proof}
 Recall that $\F_{q^3}^*=C_0\cdot\F_q^*$. When $a=0$ and $b\neq 0$,
 $$
  \psi_{0,b}(\F_{q}^*O_{(1,z)})=\sum_{\theta\in\F_q^*}\sum_{\mu\in C_0}\psi_{\F_{q^3}}(b\theta\mu)=\sum_{x\in\F_{q^3}^*}\psi_{\F_{q^3}}(bx)=-1.$$
 The computations in the case when $a\neq 0$ and $b=0$ are similar. This completes the proof.
\end{proof}

From now on, we assume that $ab\neq 0$.  Let $\chi$ be a generator of the multiplicative character group of $\F_{q^3}^*$, and set
\[
 \chi_1:=\chi^{q-1},\quad \chi_2:=\chi^{N}.
\]
The orders of $\chi_1$ and $\chi_2$ are $N$ and $q-1$, respectively. For a subset $Y$ of $\F_{q^3}^*$ (possibly a multiset) and a multiplicative character $\chi^i$, we write $\chi^i(Y):=\sum_{x\in Y}\chi^i(x)$. It is well known that
\begin{equation}\label{sumC0}
 \chi^{k}(C_0)=\begin{cases}N,\quad &\textup{if }k\equiv 0\pmod{N}{\color{red},}\\0,\quad &\textup{otherwise{\color{red},}}\end{cases}
\end{equation}
and
\begin{equation}\label{sumFqs}
 \chi^{k}(\F_q^*)=\begin{cases}q-1,\quad &\textup{if }k\equiv 0\pmod{q-1}{\color{red},}\\0,\quad &\textup{otherwise{\color{red}.}}\end{cases}
\end{equation}

We introduce two auxiliary exponential sums:
\[
 S_1=\frac{1}{q^3-1}\sum_{\ell=0}^{N-1}G(\chi_1^{-\ell})^2\chi_1^\ell(ab)\chi_1^\ell(E) \label{1.1}
\]
and
\[
 S_2=\frac{1}{q^3-1}\sum_{i=1}^{q-2}\sum_{\ell=0}^{N-1}G(\chi_2^{i}\chi_1^{-\ell})G(\chi_2^{-i}\chi_1^{-\ell})\chi_1^\ell(ab)\chi_2^{i}(ab^{-1})\chi_2^{i}\chi_1^\ell(E).
\]

\begin{lemma}
It holds that $\psi_{a,b}(D)=S_1+S_2$.
\end{lemma}
\begin{proof}
 By Lemma~\ref{orthchar},  we have
 \begin{align}
  \psi_{a,b}(\F_{q}^*O_{(1,z)})= & \sum_{\theta\in\F_q^*}\sum_{\mu\in C_0}\psi_{\F_{q^3}}(b\theta\mu+a\theta\mu^{-1}z)\nonumber                                                                                   \\
  =                              & \frac{1}{(q^3-1)^2}\sum_{i=0}^{q^3-2}\sum_{j=0}^{q^3-2}\sum_{\theta\in\F_q^*}\sum_{\mu\in C_0}G(\chi^{-i})\chi^i(b\theta\mu)G(\chi^{-j})\chi^j(a\theta\mu^{-1}z)\nonumber      \\
  =                              & \frac{1}{(q^3-1)^2}\sum_{i=0}^{q^3-2}\sum_{j=0}^{q^3-2}\sum_{\theta\in\F_q^*}G(\chi^{-i})G(\chi^{-j})\chi^i(b\theta)\chi^j(a\theta z)\sum_{\mu\in C_0}\chi^{i-j}(\mu)\nonumber \\
  \overset{\eqref{sumC0}}{=}     & \frac{N}{(q^3-1)^2}\sum_{j=0}^{q^3-2}\sum_{h=0}^{q-2}\sum_{\theta\in\F_q^*}G(\chi^{-j-Nh})G(\chi^{-j})\chi^{j+Nh}(b\theta)\chi^j(a\theta z)\nonumber                           \\
  =                              & \frac{N}{(q^3-1)^2}\sum_{j=0}^{q^3-2}\sum_{h=0}^{q-2}G(\chi^{-j-Nh})G(\chi^{-j})\chi^{j+Nh}(b)\chi^j(a z)\chi^{2j+Nh}(\F_q^*).\nonumber                                        %\label{eq:c02}
 \end{align}
 By \eqref{sumFqs}, $\chi^{2j+Nh}(\F_q^*)=q-1$ if and only if $2j+Nh\equiv 0\pmod{q-1}$, i.e., $h\equiv \frac{2j(q-2)}{3}\pmod{q-1}$. We thus have
 \[\psi_{a,b}(\F_{q}^*O_{(1,z)})
  =\frac{1}{q^3-1}\sum_{j=0}^{q^3-2}G(\chi^{-j-\frac{2j(q-2)N}{3}})G(\chi^{-j})\chi^{j+\frac{2j(q-2)N}{3}}(b)\chi^j(a z).
 \]
 %The character sum over the union of orbits $O_{(1,x)}$ for $x\in E$ is 5computed as follows.
 %\[\sum_{x\in E}\psi_{a,b}(\F_{q}^*O_{(1,x)})
 %=\frac{1}{q^3-1}\sum_{j=0}^{q^3-2}G(\chi^{-j-\frac{2j(q-2)N}{3}})G(\chi^{-j})\chi^{j+\frac{2j(q-2)N}{3}}(b)\chi^j(a)\sum_{x\in E}\chi^j(x)
 %\]
 Since $N=q^2+q+1$ and $\gcd(N,q-1)=1$, any integer $j\in[0, q^3-2]$ can be uniquely written as $j=Ni+(q-1)\ell$ for some integers $0\le i\le q-2$ and $0\le \ell\le q^2+q$ by the Chinese Remainder Theorem. Moreover, with $j=Ni+(q-1)\ell$, we have
 \[
  j+\frac{2j(q-2)N}{3}\equiv-Ni+(q-1)\ell\pmod{q^3-1}
 \]
 by the fact $\frac{(q-2)}{3}N\equiv -1\pmod{q-1}$. By rewriting  $j$ as $Ni+(q-1)\ell$ we deduce that
 \begin{align}
  \psi_{a,b}(\F_{q}^*O_{(1,z)})= & \frac{1}{q^3-1}\sum_{i=0}^{q-2}\sum_{\ell=0}^{N-1}G(\chi_2^{i}\chi_1^{-\ell})G(\chi_2^{-i}\chi_1^{-\ell})\chi_2^{-i}\chi_1^\ell(b)\chi_2^{i}\chi_1^\ell(az) \nonumber  \\
  =                              & \frac{1}{q^3-1}\sum_{\ell=0}^{N-1}G(\chi_1^{-\ell})^2\chi_1^\ell(b)\chi_1^\ell(az) \nonumber                                                                           \\
                                 & +\frac{1}{q^3-1}\sum_{i=1}^{q-2}\sum_{\ell=0}^{N-1}G(\chi_2^{i}\chi_1^{-\ell})G(\chi_2^{-i}\chi_1^{-\ell})\chi_2^{-i}\chi_1^\ell(b)\chi_2^{i}\chi_1^\ell(az).\nonumber
 \end{align}
 Taking summation over $z\in E$, we get the desired equality $\psi_{a,b}(D)=S_1+S_2$.
\end{proof}

We now explicitly evaluate $S_1$ and $S_2$. Write $ab=w^{Ns_0+(q-1)t_0}$ for some $s_0\in\{0,1,\ldots,q-2\}$ and $t_0\in \{0,1,\ldots,q^2+q\}$.

\begin{lemma}\label{abne0p1}
It holds that
 $$S_1=\begin{cases}
   \displaystyle\frac{(q+1)(q^3-q^2+1)}{3(q-1)}, & \emph{if~}\emph{Tr}_{q^3/q}(w^{(q-1)t_0})=0, \\
   \displaystyle-\frac{(q+1)^2}{3},              & \emph{otherwise}.
  \end{cases}$$
\end{lemma}
\begin{proof}
 From the proof of Proposition \ref{keyodd}, we see that $E=\bigcup_{x\in L_0}x L_x$, where each $L_x$ is a subset of  $\F_q^*$ of size $\frac{q+1}{3}$. By Lemma~\ref{chiL0}, for any $1\le\ell\le N-1$ we have
 \begin{align}
  \chi_1^\ell(E)
  =\sum_{x\in L_0}\sum_{y\in L_x}\chi_1^\ell(xy)=\frac{q+1}{3}\sum_{x\in L_0}\chi_1^\ell(x)=\frac{(q+1)}{3q}G(\chi_1^\ell).\nonumber%
  %\label{eq:Ei1}
 \end{align}
 Together with the fact $G(\chi_1^{-\ell})G(\chi_1^{\ell})=q^3$ for $1\le\ell\le N-1$, we have
 %It follows that,
 \begin{align}
  S_1= &                                                                                                                                             %\frac{1}{q^3-1}\sum_{\ell=0}^{N-1}G(\chi_1^{-\ell})G(\chi_1^{-\ell})\chi_1^\ell(ab)\sum_{x\in E}\chi_1^\ell(x)\nonumber \\
  %=&
  \frac{|E|}{q^3-1}+\frac{1}{q^3-1}\sum_{\ell=1}^{N-1}G(\chi_1^{-\ell})^2\chi_1^\ell(ab)\chi_1^\ell(E)\nonumber                                      \\
  =    & \frac{(q+1)^2}{3(q^3-1)}+\frac{q+1}{3q(q^3-1)}\sum_{\ell=1}^{N-1}G(\chi_1^{-\ell})G(\chi_1^{-\ell})G(\chi_1^{\ell})\chi_1^\ell(ab)\nonumber \\
  =    & \frac{(q+1)^2}{3(q^3-1)}+\frac{(q+1)q^2}{3(q^3-1)}\sum_{\ell=1}^{N-1}G(\chi_1^{-\ell})\chi_1^\ell(ab).\nonumber                             %\label{eq:S11}
 \end{align}
 By Lemma~\ref{partialGS} and the fact that $G(\chi_1^0)=-1$,
 we have
 \[
  S_1=\frac{(q+1)^2}{3(q^3-1)}+\frac{(q+1)q^2}{3(q-1)}
  \left( \psi_{\F_{q^3}}(ab\F_q^*)+\frac{1}{N}\right).
 \]
 We compute
 \begin{equation}\label{eq:sub1}
  \psi_{\F_{q^3}}(ab\F_q^*)=\psi_{\F_q}(w^{Ns_0}\Tr_{q^3/q}(w^{(q-1)t_0})\F_q^*)=\begin{cases}
   q-1, & \text{if~}\Tr_{q^3/q}(w^{(q-1)t_0})=0, \\
   -1,  & \text{otherwise}.
  \end{cases}
 \end{equation}
 Hence we finally obtain
 \[
  S_1=%&\frac{q+1}{3(q-1)}+\frac{(q+1)q^2N}{3(q^3-1)}\sum_{k=0}^{q-2}\psi_{\F_{q^3}}(w^{Nk+(q-1)t_0})\\
  %=&\frac{q+1}{3(q-1)}+\frac{(q+1)q^2}{3(q-1)}\sum_{k=0}^{q-2}\psi_{\F_q}(w^{Nk}\Tr_{q^3/q}(w^{(q-1)t_0}))\\
  \begin{cases}
   \frac{(q+1)(q^3-q^2+1)}{3(q-1)}, & \text{if~}\Tr_{q^3/q}(w^{(q-1)t_0})=0, \\
   -\frac{(q+1)^2}{3},              & \text{otherwise}.
  \end{cases}
 \]
 This completes the proof of the lemma.
\end{proof}
%\vspace{0.3cm}

Next we evaluate $S_2$. By the definition of $E$, we have
$$ \chi_2^i\chi_1^\ell(E)=\frac{1}{3}\left(
 \chi_2^i\chi_1^\ell(D_1)+ \chi_2^i\chi_1^\ell(D_2)- \chi_2^i\chi_1^\ell(T_0)
 \right).$$
Therefore $S_2=\Sigma_1+\Sigma_2+\Sigma_3$, where
\begin{align*}
 \Sigma_1 & =\frac{1}{3(q^3-1)}\sum_{i=1}^{q-2}\sum_{\ell=0}^{N-1}G(\chi_2^{i}\chi_1^{-\ell})G(\chi_2^{-i}\chi_1^{-\ell})\chi_1^\ell(ab)\chi_2^{i}(ab^{-1})\chi_2^{i}\chi_1^\ell(D_1),  \\
 \Sigma_2 & =\frac{1}{3(q^3-1)}\sum_{i=1}^{q-2}\sum_{\ell=0}^{N-1}G(\chi_2^{i}\chi_1^{-\ell})G(\chi_2^{-i}\chi_1^{-\ell})\chi_1^\ell(ab)\chi_2^{i}(ab^{-1})\chi_2^{i}\chi_1^\ell(D_2),  \\
 \Sigma_3 & =\frac{-1}{3(q^3-1)}\sum_{i=1}^{q-2}\sum_{\ell=0}^{N-1}G(\chi_2^{i}\chi_1^{-\ell})G(\chi_2^{-i}\chi_1^{-\ell})\chi_1^\ell(ab)\chi_2^{i}(ab^{-1})\chi_2^{i}\chi_1^\ell(T_0). \\
\end{align*}

%%%%%%%%%%%%%%%%%%%%%%%%%%%%%%%%%%%%%%%%%%%%%%%%%%%%%%%%%%
%%%%%%%%%%%%%%%%%%%%%%%%%%%%%%%%%%%%%%%%%%%%%%%%%%%%%%%%%%

Write $ab=w^{Ns_0+(q-1)t_0}$, $ab^{-1}=w^{Nu_0+(q-1)v_0}$, and set $z_0=w^{Nu_0+(q-1)t_0}$ and $z_1=w^{-Nu_0+(q-1)t_0}$. The next proposition gives the values of $\Sigma_1$, $\Sigma_2$ and $\Sigma_3$.
\begin{proposition}\label{odd1.2}
 For each $z\in T_0$, define
 \begin{align*}
  \mu_z:=  & |\{(y,\lambda)\in C_0\times \F_q: y-(z^q-z^{q^2})+\lambda=0\}|,      \\
  \mu_z':= & |\{(y,\lambda)\in C_0\times \F_q: y-\beta(z^q-z^{q^2})+\lambda=0\}|,
 \end{align*}
where $\beta=-3^{-1}\in \F_q$. Then it holds that
 \begin{align*}
  (\Sigma_1,\Sigma_2,\Sigma_3)=\begin{cases}
   (\frac{q^3}{3}\mu_{z_0}-\frac{q^4}{3(q-1)},\,\frac{q^3}{3}\mu_{z_1}'-\frac{q^4}{3(q-1)},\,0), & \emph{if~}w^{(q-1)t_0}\in T_0, \\
   (0,\,0,\,0),                                                                                  & \emph{otherwise.}
  \end{cases}
 \end{align*}
\end{proposition}

The proof of Proposition~\ref{odd1.2} involves very complicated computations of
exponential sums. To streamline the presentation of the paper, we delay the proof to the Appendix.

\begin{proof}[Proof of Theorem~\ref{mtheven}]
 As mentioned in Subsection \ref{subsec_M}, it suffices to establish \eqref{toprove}
 for nonzero element $(a,b)$ of $V$, i.e.,
 \begin{equation}
  \psi_{a,b}(D)=\begin{cases}
   -\frac{(q+1)^2}{3}+q^3, & \text{if~}(a,b)\in D, \\
   -\frac{(q+1)^2}{3},     & \text{otherwise}.\end{cases}
 \end{equation}

 If either $a=0$ or $b=0$, then $(a,b)\not\in D$ and the claim has been established in  Lemma~\ref{ab0}.

 We next treat the case when $ab\neq 0$. Write $ab=w^{Ns_0+(q-1)t_0}$, $ab^{-1}=w^{Nu_0+(q-1)v_0}$, and set $x_0=w^{(q-1)t_0}$, $\theta_0=w^{Nu_0}$. It follows that $a^2=w^{N(u_0+s_0)+(q-1)(t_0+v_0)}$. In the rest of this proof, we use the notation introduced in the statement of  Proposition~\ref{odd1.2}. In particular, $z_0=x_0\theta_0$ and $z_1=x_0\theta_0^{-1}$.

 We claim that $(a,b)\in D$ if and only if $z_1\in E$. We write $a=x\mu$ for some $x\in \F_q^\ast$ and $\mu\in C_0$. From $a^2=w^{N(u_0+s_0)+(q-1)(t_0+v_0)}$ we deduce that $x^2=w^{N(u_0+s_0)}$. It is straightforward to check that $ab=x^2z_1$, so $b=(x\mu)^{-1}x^2z_1=x\mu^{-1}z$. That is, $(a,b)=(x\mu,x\mu^{-1}z_1)$. The claim now follows from the definition of $D$.

 By Lemma~\ref{abne0p1} and Proposition~\ref{odd1.2}, we have
 \begin{align*}
  \psi_{a,b}(D)=\begin{cases}
   \frac{(q+1)(q^3-q^2+1)}{3(q-1)}+\frac{q^3}{3}(\mu_{z_0}+\mu_{z_1}')-\frac{2q^4}{3(q-1)}, & \text{if~}x_0\in T_0, \\
   -\frac{(q+1)^2}{3},                                                                      & \text{otherwise}.
  \end{cases}
 \end{align*}
 Hence, we  need to compute $\mu_{z_0}+\mu_{z_1}'$ under the assumption that $x_0\in T_0$. Set $\tilde{z}_{0}:=x_0^q-x_0^{q^2}$. We now have
 \begin{align*}
  \mu_{z_0}+\mu_{z_1}'= & |\{(y,\lambda)\in C_0\times\F_q : y=(z^q-z^{q^2})-\lambda\}|+|\{(y,\lambda)\in C_0\times\F_q : y=\beta (z^q-z^{q^2})-\lambda\}|                      \\
  =                     & |\{\lambda\in\F_q: \N_{q^3/q}(\lambda+\theta_0\tilde{z}_0)=1\}|+|\{\lambda\in\F_q: \N_{q^3/q}(\lambda+\theta_0^{-1}\beta\tilde{z}_0)=1\}|            \\
  =                     & |\{\lambda\in\F_q: \N_{q^3/q}(\lambda+\tilde{z}_0)=\theta_0^{-3}\}|+|\{\lambda\in\F_q: \beta^3\N_{q^3/q}(\lambda+\tilde{z}_0)^{-1}=\theta_0^{-3}\}|,
 \end{align*}
 which is the multiplicity of $\theta_0^{-3}$ in the multiset $W_{x_0}$. By Proposition~\ref{keyodd}, this multiplicity is equal to $1$ or $4$, and correspondingly $ \psi_{a,b}(D)=-\frac{(q+1)^2}{3}$ or $q^3-\frac{(q+1)^2}{3}$. Moreover $\mu_{z_0}+\mu_{z_1}'=4$ if and only if $\theta_0^{-1}\in L_{x_0}$, i.e., $z_1=x_0\theta_0^{-1}\in E$. To sum up, we have shown that
 \begin{align*}
  \psi_{a,b}(D)=\begin{cases}
   q^3-\frac{(q+1)^2}{3}, & \text{if~} z_1\in E, \\
   -\frac{(q+1)^2}{3},    & \text{otherwise}.
  \end{cases}
 \end{align*}
 Also we have shown that $ (a,b)\in D$ if and only if $z_1\in E$. The proof is now complete.
\end{proof}

\section{The stabilizer of ${\mathcal M}$ in $\textup{P}\Omega(V)$}

Recall that $V=\F_{q^3}\times \F_{q^3}$, and $Q: V\to\F_q$ defined by $Q((x,y))=\Tr_{q^3/q}(xy),\ \forall (x,y)\in V$, is a nonsingular hyperbolic quadratic form on $V$. Let $\Omega(V)$ be the derived subgroup of the isometry group of the quadratic space $(V,\,Q)$, and let $\textup{P}\Omega(V)$ be the quotient group modulo its center. In this section, we determine the stabilizer of $\M$ in $\textup{P}\Omega(V)$.   

Set $W:=\{x\in\F_{q^3}:\,\Tr_{q^3/q}(x)=0\}$, $U_1:=\F_{q^3}\times\{0\}$, and $U_2:=\{0\}\times\F_{q^3}$. Let $E$ be the subset of $\F_{q^3}^*$ as in Proposition \ref{keyodd}. For each $x\in L_0$, there exists a subset $L_x$ of $\F_q^*$ of size $\frac{q+1}{3}$ such that $E=\cup_{x\in L_0}xL_x$ by the proof of Proposition \ref{keyodd}.

Let $\square$ be the set of squares of $\F_q^*$ (so in the case when $q$ is even, $\square=\F_q^*$). For two subsets $A,\,B$ of $\F_{q^3}^*$, define $A\cdot B:=\{ab:\,a\in A,\,b\in B\}$. In particular, if $A=\{a\}$, we write $aB$ for $A\cdot B$.

\begin{lemma}\label{lem_W}
 As $\F_q$-vector spaces, $\F_{q^3}=W\oplus\F_q$. Furthermore $\F_q^*\cdot E=W\setminus\{0\}$.
\end{lemma}
\begin{proof}
 Since $\gcd(3,q-1)=1$, $W$ and $\F_q$ intersect trivially, and  the first claim follows. The second is clear from the above description of $E$.
\end{proof}

For each $u\in L_0$, we define $B_u:=\{y^{q^2}u^q-y^qu^{q^2}:\,y\in L_0\setminus\{u\}\}$. It is routine to check that the elements of $B_u$ lie in $\F_q^*$ by the fact $\Tr_{q^3/q}(y)=0$ for $y\in L_0$.
\begin{lemma}\label{lem_Bu}
 For $u\in L_0$, we have $|B_u|=\frac{2}{3}(q+1)-1$ and $B_u\cdot L_u=\F_q^*$.
\end{lemma}
\begin{proof}
 We have shown that $W$ and $\F_q$ intersect trivially, so $u\not\in\F_q$ and $W=\< u,\,u^q\>_{\F_q}$. Moreover, $u^{q-1}\not\in\F_q^*$ by the fact that $\gcd(q-1,q^2+q+1)=1$. We thus have $L_0=\left\{\frac{u^q+\lambda u}{\N_{q^3/q}(u^q+\lambda u)^{1/3}}:\,\lambda\in\F_q\right\}\cup\{u\}$. It follows that $B_u=\left\{\frac{u^{q+1}-u^{2q^2}}{\N_{q^3/q}(u^q+\lambda u)^{1/3}}:\,\lambda\in\F_q\right\}$, and its size equals that of $\{\N_{q^3/q}(u^{q-1}+\lambda):\,\lambda\in\F_q\}$. Write $x:=\lambda+\frac{1}{3}\Tr_{q^3/q}(u^{q-1})\in\F_q$. Then
 \[
  \N_{q^3/q}(u^{q-1}+\lambda)=x^3+bx+c
 \]
 for some $b,c\in\F_q$. The polynomial $X^3+bX+c\in \F_q[X]$ has no roots in $\F_q$, so it is irreducible over $\F_q$. It follows that $bc\ne 0$. The polynomial $X^3+bX$ is a Dickson polynomial of degree $3$, and its value set over $\F_q$ has size $\frac{q-1}{2\cdot\gcd(3,q-1)}+\frac{q+1}{2\cdot\gcd(3,q+1)}=\frac{2}{3}(q+1)-1$ by Theorem 10 and Theorem 10' of \cite{valueset}. For any $a\in\F_q^*$, we have
 \[
  |\{ay^{-1}:\,y\in L_u\}\cap B_u|\ge |L_u|+|B_u|-|\F_q^*|\ge 1.
 \]
 This show that $a\in B_u\cdot L_u$. Hence $B_u\cdot L_u=\F_q^*$. The proof of the lemma is now  complete.
\end{proof}

The generators of $(V,\,Q)$ fall into two equivalence classes; two generators $U$ and $U'$ are equivalent if and only if $U\cap U'$ has dimension $1$, cf. \cite[Theorem 1.39]{GGG}. The group $\Omega(V)$  stabilizes each equivalence class, cf. \cite[p. 30]{KL129}. The two subspaces $U_1$ and $U_2$ are both generators of the quadratic space $(V,\,Q)$, and they are in different equivalence classes.
\begin{lemma}\label{lem_stabgen}
 The only generators of $(V, Q)$ that are disjoint from $\M$ are $U_1$ and $U_2$.
\end{lemma}
\begin{proof}
 It is clear that $U_1$ and $U_2$ are disjoint from $\M$. Suppose that $U$ is a generator other than $U_1,\,U_2$. We will show that $U$ intersects $\M$ nontrivially.

 We first consider the case where $U$ and $U_2$ are equivalent. In this case, $U\cap U_2$ is a projective point $P$. By applying the action of some element in $i(C_0)\leq PGO^+(6,q)$ if necessary, we may assume without loss of generality that $P=\< (0,1)\>$. It is clear that $P^\perp=W\times \F_{q^3}$. Since $\F_{q^3}=W\oplus\F_q$, we identify $W\times W$ with $P^\perp/P$ naturally. In this way, $W\times W$ becomes a quadratic space $\Q^+(3,q)$ whose inherited quadratic form is the same as the restriction of $Q$ to $W\times W$.
 We have $\M\cap P^\perp=\{\<(y,\,y^{-1}x)\>:\,x\in E,\,y\in L_0\}$, and the corresponding set in $W\times W$ is $\M_P:=\{\<(y,\,\tau_y(x))\>:\,x\in E,\,y\in L_0\}$, where $\tau_y(x):=y^{-1}x-\frac{1}{3}\Tr_{q^3/q}(y^{-1}x)$. It is straightforward to check that
 \[
  \ker(\tau_y)=\F_q\cdot y,\;\textup{Im}(\tau_y)\le W,\;\tau_y(W)\le\{z\in W:\, \Tr_{q^3/q}(yz)=0\}=\< y^q-y^{q^2}\>_{\F_q}.
 \]
 We thus have $\tau_y(W)= \< y^q-y^{q^2}\>_{\F_q}$ by comparing dimensions. Let $U'$ be the totally singular line of $W\times W$ corresponding to $U$. To show that $U$ intersects $\M$ nontrivially, it suffices to show that $U'$ intersects $\M_P$ nontrivially. There are $2(q+1)$ totally singular lines of $\Q^+(3,q)$, these are  $\ell_y=\< (y,0),\,(0,y^q-y^{q^2})\>$ with $y\in L_0$, $\ell_a'=\{\< (x,ax^q-ax^{q^2})\>:\,x\in W\}$ with $a\in\F_q$ and $\ell_\infty'=\{0\}\times W$. The last line $\ell_\infty'$ corresponds to the generator $U_2$, so $U'\ne\ell_\infty'$.
 
 \begin{enumerate}
  \item[(1)] If  $U'=\ell_y$ for some $y\in L_0$,  then the point $\<(y,\,\tau_y(x))\>$ with $x\in E$ is in $U'$ if there exists $\lambda\in\F_q^*$ such that $\tau_y(\lambda x)=y^q-y^{q^2}$. By Lemma \ref{lem_W}, we have $\F_q^*\cdot E=W\setminus\{0\}$. The existence of such $x\in E$ and $\lambda\in\F_q^*$ now follows from the fact that $\tau_y(W)= \< y^q-y^{q^2}\>_{\F_q}$.
  \item[(2)] If $U'=\ell_a'$ for some fixed $a\in\F_q$, then  $U'\cap\M_P\ne\emptyset$ if there exists $y\in L_0$, $u\in L_0$ and $c\in L_u$ such that
        $ \tau_y(uc)=ay^q-ay^{q^2}$.  The left hand side equals $z^q-z^{q^2}$ with $z=-\frac{1}{3}(y^{-1}u)^qc+\frac{1}{3}(y^{-1}u)^{q^2}c$, so  $ay-z\in\F_q$. By taking the relative trace, we see that it equals $0$. By the fact that $\N_{q^3/q}(y)=1$ for $y\in L_0$, we deduce that $-3a=(y^{q^2}u^q-y^qu^{q^2})c$. When $a=0$, we can simply take $y=u\in L_0$ and $c\in L_u$ arbitrarily. When $a\ne 0$, we take $u$ to be any element of $L_0$ and the existence of the desired $(y,c)$ pair follows from Lemma \ref{lem_Bu}.
 \end{enumerate}
 In both cases, we have shown that $U'$ intersects $\M_P$ nontrivially. This establishes the claim in the case $U$ is in the same equivalence class as $U_2$.

 We next consider the case where $U$ and $U_1$ are equivalent. Observe that $O_{(1,xa)}=O_{(xa^{-1},1)}$ for $x\in C_0$ and $a\in \F_q^*$, so $\M=\cup_{x\in E'}O_{(x,1)}$, where $E'=\cup_{x\in L_0}xL_x'$ with $L_x'=\{a^{-1}:\,a\in L_x\}$. The argument is exactly the same as in the previous case.
\end{proof}

Let $K$ be the stabilizer of $U_1$ and $U_2$ in $\Omega(V)$, i.e., $K=\{\alpha\in \Omega(V) : \alpha(U_1)=U_1\;{\rm and}\; \alpha(U_2)=U_2\}$. By \cite[Lemma 4.1.9]{KL129}, $K$  consists of
\[
 \kappa(h,h^*):\,V\rightarrow V,\;(x,y)\mapsto (h(x),h^*(y)),
\]
where both $h$ and $h^*$ are bijective $\F_q$-linear transformations of $\F_{q^3}$ such that $\det(h),\det(h^*)\in\square$ and $Q((x,y))=Q((h(x),h^*(y)))$ for all $x,\,y\in\F_{q^3}$. Here, $\det(h)$ is the determinant of $h$ with respect to any $\F_q$-basis of $\F_{q^3}$. For each bijective $\F_q$-linear transformation $h$ of $\F_{q^3}$ with $\det(h)\in\square$, there is a unique $h^*$ such that $\kappa(h,h^*)\in K$, and vice versa.

We now describe some special elements of $K$. For $a\in\F_{q^3}^*$, define
\[
 h_a:\,\F_{q^3}\rightarrow\F_{q^3},\quad x\mapsto ax,
\]
and set $\kappa_a:=\kappa(h_a,h_{a^{-1}})$. An element $z\in C_0$ (which we identify with the corresponding element in $i(C_0))$ acts on $V$ in exactly the same way as $\kappa_z$.

\begin{lemma}\label{lem_ha}
 For $a\in\F_{q^3}^*$, $\kappa_a$ is in $K$ if and only if $a$ is a square in $\F_{q^3}^*$.
\end{lemma}
\begin{proof}
 The linear transformation $\kappa_a$ clearly has determinant $1$ and stabilizes the generators $U_1$ and $U_2$, so it suffices to show that $\det(h_a)\in\square$ if and only if $a$ is a square in $\F_{q^3}^*$. For $a\in C_0$, we have $h_a^{q^2+q+1}=\textup{id}_{\F_{q^3}}$, so $\det(h_a)^{q^2+q+1}=1$. It follows that $\det(h_a)=1$ from the fact $\gcd(q^2+q+1,q-1)=1$. For $a\in\F_q^*$, we have $\det(h_a)=a^3$, which is a square if and only if $a$ is. The claim then follows readily from the fact that $\F_{q^3}^*=C_0\cdot\F_q^*$.
\end{proof}

We define $\iota:\,K\rightarrow \textup{PGL}(3,q)$ such that $\iota(g)$ is the quotient image of $g|_{U_1}$ in $\textup{PGL}(3,q)$, where $g|_{U_1}$ is the restriction of $g$ to $U_1$. Since $\gcd(3,q-1)=1$, we have $\textup{PGL}(3,q)=\textup{PSL}(3,q)$. The homomorphism $\iota$ is surjective by the above description of $K$.
\begin{lemma}\label{lem_keriota}
 We have $\ker(\iota)=\kappa_\square$, where $\kappa_\square=\{\kappa_a:\,a\in\square\}$.
\end{lemma}
\begin{proof}
 If $\kappa=\kappa(h,h^*)\in\ker(\iota)$, then $h=h_a$ for some $a\in\F_q^*$ and correspondingly $\kappa=\kappa_a$. The claim is now an easy consequence of Lemma \ref{lem_ha}.
\end{proof}

Let $\sigma$ be the $\F_q$-linear transformation of $V$ such that $\sigma((x,y))=(x^q,y^q)$. It has order $3$ and stabilizes both $U_1$ and $U_2$.
\begin{lemma}
 We have $\sigma\in K$, and $\sigma(\M)=\M$.
\end{lemma}
\begin{proof}
 The first claim follows by the same argument as in the proof of Lemma \ref{lem_ha}. The second claim is equivalent to $\sigma(E)=E$, or equivalently, $L_{x^q}=\{a^q:\,a\in L_x\}$ for each $x\in L_0$. This is clear from the definition of $L_x$ in the proof of Proposition \ref{keyodd}.
\end{proof}

Let $G$ be the stabilizer of $\M$ in $\Omega(V)$.  Let $\alpha\in G$. From $U_1\cap {\mathcal M}=\emptyset$ and $U_2\cap {\mathcal M}=\emptyset$, we obtain $\alpha(U_1)\cap {\mathcal M}=\emptyset$ and $\alpha(U_2)\cap {\mathcal M}=\emptyset$. By Lemma \ref{lem_stabgen} and the fact that $U_1$ and $U_2$ are in different equivalence classes, we deduce that $\alpha(U_1)=U_1$ and $\alpha(U_2)=U_2$, and so $\alpha\in K$. We have shown that $G\leq K$. Moreover, $G$ contains the subgroup $H$ generated by $\sigma$ and $i(C_0)$.
\begin{lemma}\label{lem_iotaG}
 The group $\iota(G)$ has order $3(q^2+q+1)$ when $q>2$.
\end{lemma}
\begin{proof}
 The group $\iota(H)$ has order $3(q^2+q+1)$ and is a maximal subgroup of $\textup{PSL}(3,q)$ by {\color{blue}\cite[Table 8.3]{BHR407}}. Hence either $\iota(G)=\textup{PSL}(3,q)$ or $\iota(G)=\iota(H)$.

 Suppose that $\iota(G)=\textup{PSL}(3,q)$. Fix an element $u\in L_0$, and take $\lambda$ to be a primitive element of $\F_q^*$. Let $g=\kappa(h,h^*)$ be the element of $K$ such that $h^*(1)=1$, $h^*(u)=\lambda u$ and $h^*(u^q)=\lambda^{-1}u^q$. We deduce that $h(1)=1$ from the property $Q((1,x))=Q((h(1),h^*(x)))$.  By our assumption there exists $a\in\square$ such that $\kappa_ag\in G$, i.e., $\kappa_a g$ stabilizes $\M$. The image of $\{\<(1,x)\>:\,x\in E\}\subseteq\M$ under $\kappa_a g$ is $\{\<(1,a^{-2}h^*(x))\>:\,x\in E\}$, so we have $E=a^{-2}h^*(E)$. Comparing both sides, we deduce that $uL_u=a^{-2}\lambda uL_u$, $u^qL_{u^q}=a^{-2}\lambda^{-1} u^qL_{u^q}$. Taking the product over the set on each side, we get $(a^{-2}\lambda)^{(q+1)/3}=1$, $(a^{-2}\lambda^{-1})^{(q+1)/3}=1$. It follows that $\lambda^{2(q+1)/3}=1$. If $q>5$, then
 $\frac{2(q+1)}{3}<q-1$, and the equality $\lambda^{2(q+1)/3}=1$ contradicts the assumption that $\lambda$ is primitive. If $q=5$, we have $a^4=1$ and this leads to $\lambda^2=1$, again contradicting the assumption that $\lambda$ is primitive. The proof is now complete.
\end{proof}

\begin{lemma}\label{lem_Lxsq}
 If $q$ is odd, then $|L_x\cap\square|=\frac{q+1}{6}$ for each $x\in L_0$.
\end{lemma}
\begin{proof}
 From the proof of Proposition \ref{keyodd}, we know that  $W_x=\F_q^*+3L_x^{(3)}$ in the group ring $\Z[\F_q^*]$, where
 $W_x$ is the same as in \eqref{eqn_Wx}. Let $\rho$ be the quadratic character of $\F_q^*$, which maps squares to $1$ and nonsquares to $-1$. Then $\rho(\F_q^*)=0$ and $\rho(W_x)=3\rho(L_x)$. Since $\gamma=-27$ is a nonsquare in $\F_q^*$, we deduce that $\rho(W_x)=0$. It follows that $\rho(L_x)=0$, i.e., $L_x$ has the same number of squares as nonsquares. This completes the proof.
\end{proof}

\begin{theorem}
 The group $G$ has order $3(q^2+q+1)s$, where $s=1$ or $s=\gcd(2,\frac{q-1}{2})$ according as $q$ is even or odd.
\end{theorem}
\begin{proof}
 The case $q=2$ is verified by Magma \cite{Magma}; so from now on we assume that $q>2$. By Lemma \ref{lem_iotaG}, $G$ lies in the group $H\times \kappa_\square$, where $\kappa_\square$ is as in Lemma \ref{lem_keriota}. We have shown that $H\le G$, so $G=H\times (G\cap\kappa_\square)$. It now suffices to determine the stabilizer of $\M$ in $\kappa_\square$.

 Suppose that $\kappa_a$ stabilizes $\M$, where $a$ is a square of $\F_q^*$. The condition $\kappa_a(\M)=\M$ is equivalent to  $a^2E=E$, i.e., $a^2L_x=L_x$ for each $x\in L_0$. Taking the product over the set on each side, we deduce that $a^{2(q+1)/3}=1$. If $q$ is even, then the order of $a$ divides $\gcd(2(q+1)/3,q-1)=1$, implying $a=1$.  If $q\equiv 3\pmod{4}$, then $\gcd(\frac{q-1}{2},\,\frac{2(q+1)}{3})=1$ and we also get $a=1$. If $q\equiv 1\pmod{4}$, then from $a^2L_x=L_x$ we deduce that $a^2(L_x\cap\square)=L_x\cap\square$. By Lemma \ref{lem_Lxsq} we have $|L_x\cap\square|=\frac{q+1}{6}$. By the same argument we get $a^{(q+1)/3}=1$. In this case, we have $\gcd(\frac{q+1}{3},\frac{q-1}{2})=2$, so $a^2=1$, i.e., $a=\pm 1$. Since $-1$ is in $\square$, we see that indeed $\kappa_{-1}$ is in $G$. This completes the proof.
\end{proof}

As a corollary, the stabilizer of $\M$ in $\textup{P}\Omega(V)$ has order $3(q^2+q+1)$. By the isomorphism $\textup{P}\Omega^+(6,q)\cong \textup{PSL}(4,q)$, we see that the stabilizer of the corresponding Cameron-Liebler line class in $\textup{PSL}(4,q)$ has  order $3(q^2+q+1)$.

\section{Concluding Remarks}

In this paper, we have constructed Cameron-Liebler line classes in $\textup{PG}(3,q)$ with parameter $x=(q+1)^2/3$ for all prime powers $q$ congruent to $2$ modulo $3$. This is a contribution to the study of the central problem about Cameron-Liebler line classes in $\PG(3,q)$. Besides the trivial examples with $x=1,2$, all known infinite families of Cameron-Liebler line classes prior to our work have parameters $x=(q^2-1)/2$ or $x=(q^2+1)/2$, up to complement.

Most notably, we have constructed the first infinite family of nontrivial Cameron-Liebler line classes in $\textup{PG}(3,q)$ with $q$ even. In contrast, the first nontrivial infinite family of Cameron-Liebler line classes in $\textup{PG}(3,q)$ for odd $q$ was constructed by Bruen and Drudge \cite{bd} twenty years ago. The major obstacle in the characteristic two case seems to be that such line classes, if they exist, tend not to be highly symmetric. In our construction, the Cameron-Liebler line classes have automorphism groups of medium sizes. This fact makes it difficult to give a neat geometric description of the objects we have constructed. Our proof is very algebraic, due to the nature of our construction.

In Section 5, we have determined the stabilizers of our Cameron-Liebler line classes in $\textup{PSL}(4,q)$. The size of the stabilizer is $3(q^2+q+1)$. It will be of particular interest to find infinite families of Cameron-Liebler line classes whose stabilizers in $\textup{PSL}(4,q)$ do not grow as $q$ increases.

\section*{Appendix: Proof of Proposition~\ref{odd1.2}}
In this appendix, we will prove  Proposition~\ref{odd1.2}. Recall that $N=q^2+q+1$. We start with an observation on Gauss sums. Let $S$ be any subset of $\F_{q^3}^*$, and set $T_S:=\{(s,t):\,0\le i\le N-1, 0\le t\le q-2, w^{s(q-1)+tN}\in S\}$. By the definition of Gauss sums, for any integers $i$ and $\ell$ and $\epsilon,\delta\in \{1,-1\}$ we have
\begin{align}
 G(\chi_2^{\epsilon i}\chi_1^{\delta \ell}) \chi_2^i\chi_1^{\ell}(S)
 = & \sum_{y\in \F_{q^3}^*}\sum_{z\in S}\chi_2^{\epsilon i}\chi_1^{\delta \ell}(y)\chi_2^i\chi_1^{\ell}(z)\psi_{\F_{q^3}}(y)\nonumber                                                    \\
 = & \sum_{y\in \F_{q^3}^*}\sum_{(s,t)\in T_S}\chi_2^{\epsilon i}(yw^{\epsilon s(q-1)+\epsilon tN})\chi_1^{\delta \ell}(yw^{ \delta s(q-1)+\delta tN})\psi_{\F_{q^3}}(y).\label{eq:rem1}
\end{align}
Since $\chi_2(w^{s(q-1)})=1$ and $\chi_1(w^{tN})=1$, continuing from \eqref{eq:rem1}, we have
\begin{align}
 G(\chi_2^{\epsilon i}\chi_1^{\delta \ell})\chi_2^{i}\chi_1^{ \ell}(S)%\nonumber\\
 = & \sum_{y\in \F_{q^3}^*}\sum_{(s,t)\in T_S}
 \chi_2^{\epsilon i}(yw^{\delta s(q-1)+\epsilon tN})\chi_1^{\delta \ell}(yw^{\delta s(q-1)+\epsilon tN})\psi_{\F_{q^3}}(y)\nonumber                  \\
 =  & \sum_{z\in\F_{q^3}^*}\sum_{(s,t)\in T_S}\chi_2^{\epsilon i}\chi_1^{\delta \ell}(z)\psi_{\F_{q^3}}(zw^{-\delta s(q-1)-\epsilon tN}).\label{tech}
\end{align}
This identity will be used in the rest of the proof.

Let $D_3=\beta D_2=[x\N_{q^3/q}(\lambda+x^q-x^{q^2})^{-\frac{1}{3}}: x\in L_0,\lambda\in \F_q]$. To evaluate $\Sigma_1$, we  need the following observation: By \eqref{tech} we have
\begin{equation}\label{tech1}
 G(\chi_2^{i}\chi_1^{-\ell}) \chi_2^{i}\chi_1^\ell(D_1)
 =\sum_{z\in \F_{q^3}^*}\chi_2^i\chi_1^{-\ell}(z)\psi_{\F_{q^3}}(zD_3)
\end{equation}
for any $1\leq i\leq q-2$ and $0\leq \ell\leq q^2+q$.

\begin{lemma}\label{chvD3}
 Let $R=\{\lambda+(h^{q^2}-h^q):\lambda\in \F_q, h\in L_0\}$.
 For $z\in\F_{q^3}^*$, it holds that
 \[
  \psi_{\F_{q^3}}(zD_3)=
  \begin{cases}
   q^2+q,                    & \emph{if }z\in \F_q^*,                        \\
   -1+\psi_{\F_{q^3}}(eC_0), & \emph{if }z\in eR\emph{ for some }e\in\F_q^*.
  \end{cases}
 \]
\end{lemma}
\begin{proof}%[Proof of Lemma~\ref{charD3}]
 We first note that $R$ is a system of coset representatives for $(\F_{q^3}^*\setminus\F_{q}^*)/\F_{q}^*$.
 Assume that there are $\lambda_1,\lambda_2\in\F_q$, $d\in \F_q^\ast$ and $h_1,h_2\in L_0$ such that $\lambda_1+(h_1^{q^2}-h_1^q)=d\lambda_2+d(h_2^{q^2}-h_2^q)$. Then by taking trace of both sides, we have $\lambda_1=d\lambda_2$. Note that $h_1^{q^2}-h_1^q=d(h_2^{q^2}-h_2^q)$ implies that $h_1^{q^2}-dh_2^{q^2}=h_1^q-dh_2^q$, i.e.,  $h_1-d h_2\in\F_q$. Hence,  we have
 \[
  0=\Tr_{q^3/q}(h_1)-d\Tr_{q^3/q}(h_2)=\Tr_{q^3/q}(h_1-dh_2)=3(h_1-dh_2),
 \]
 which implies that $h_1=dh_2$. By taking norm of both sides, we have $d=1$, $\lambda_1=\lambda_2$ and $h_1=h_2$. %This shows that $R$ is a set of size $q(q+1)$.
 It is clear that none of the elements of $R$ is in $\F_q^*$. Hence $R$ is a system of coset representatives of $(\F_{q^3}^*\setminus\F_{q}^*)/\F_{q}^*$.

 Next we evaluate $\psi_{\F_{q^3}}(zD_3)$. Let $\eta_{q-1}$ be a fixed multiplicative character of  order $q-1$ of $\F_{q}$. Then,  %$\psi_{\F_{q^3}}(z D_3)$ is given by
 we have
 \begin{align}
  \psi_{\F_{q^3}}(zD_3)= & \sum_{\lambda\in\F_q}\sum_{x\in L_0}\psi_{\F_{q^3}}(zx\N_{q^3/q}(\lambda+x^{q}-x^{q^2})^{-1/3})\nonumber \\
  =                      & \frac{1}{q-1}\sum_{c\in \F_q^\ast}
  \sum_{\lambda\in \F_q}
  \sum_{x\in L_0}\sum_{i=0}^{q-2}
  \psi_{\F_{q^3}}(zxc^{-1})
  \eta_{q-1}^i((\lambda+x^q-x^{q^2})^{N})\eta_{q-1}^{-3i}(c)\nonumber                                                               \\
  =                      & \frac{1}{q-1}\sum_{c\in \F_q^\ast}
  \sum_{\lambda\in \F_q}
  \sum_{x\in L_0}\sum_{i=1}^{q-2}
  \psi_{\F_{q^3}}(zxc^{-1})
  \eta_{q-1}^i((\lambda+x^q-x^{q^2})^{N})\eta_{q-1}^{-3i}(c)\label{eq1}                                                             \\
                         & \,+\frac{1}{q-1}\sum_{c\in \F_q^\ast}
  \sum_{\lambda\in \F_q}
  \sum_{x\in L_0}
  \psi_{\F_{q^3}}(zxc^{-1}). \label{eq1R}
 \end{align}
 Denote the summands in \eqref{eq1} and \eqref{eq1R} by $W_1$ and $W_2$, respectively.
 Then,  $\psi_{\F_{q^3}}(zD_3)=W_1+W_2$.
 Here, it is easy to see that
 \begin{align*}
  W_2 & =\frac{1}{q-1}\sum_{c\in \F_q^\ast}
  \sum_{d\in \F_q}
  \sum_{x\in C_0}
  \psi_{\F_{q^3}}(zxc^{-1}+xd)              \\
      & =\frac{1}{q-1}
  \sum_{d\in \F_q}
  \sum_{x\in \F_{q^3}^\ast}
  \psi_{\F_{q^3}}(x(z+d))=\frac{1}{q-1}
  \begin{cases}
   q^3-q, & \mbox{ if $z\in \F_q$},     \\
   -q,    & \mbox{ if $z\not\in \F_q$}.
  \end{cases}
 \end{align*}
 We next evaluate $W_1$. Let $\rho_{q-1}$ be the lift of $\eta_{q-1}$ to $\F_{q^3}^*$, i.e., $\rho_{q-1}(x)=\eta_{q-1}(x^{N})$. We note that for any $s\in \F_{q}^*$, $\rho_{q-1}(s)=\eta_{q-1}(s^{N})=\eta_{q-1}(s^3)$.
 Then
 \begin{equation*}%\label{eq2}
  %\eqref{eq1}
  W_1=\frac{1}{q-1}\sum_{c\in \F_q^\ast}
  \sum_{\lambda\in \F_q}
  \sum_{x\in L_0}\sum_{i=1}^{q-2}
  \psi_{\F_{q^3}}(zxc^{-1})
  \rho_{q-1}^i(\lambda+x^q-x^{q^2})\rho_{q-1}^{-i}(c).
 \end{equation*}
 By Lemma~\ref{orthchar}, we have
 \[
  \rho_{q-1}^i(\lambda+x^q-x^{q^2})=\frac{G(\rho_{q-1}^i)}{q^3}\sum_{b\in  \F_{q^3}^\ast}
  \psi_{\F_{q^3}}(b(\lambda+x^q-x^{q^2}))\rho_{q-1}^{-i}(-b), \quad 1\leq i\leq q-2.
 \]
 Substituting $\rho_{q-1}^i(\lambda+x^q-x^{q^2})$ in the expression for $W_1$ by the right-hand-side expression of the above equation, we have
 \begin{align}
  W_1= & \frac{1}{q^3(q-1)}\sum_{c,d\in \F_q^\ast}
  \sum_{\lambda\in \F_q}
  \sum_{x\in L_0}\sum_{i=0}^{q-2}\sum_{h\in C_0}
  G(\rho_{q-1}^i)\psi_{\F_{q^3}}(zxc^{-1}+hd(\lambda+x^q-x^{q^2}))\rho_{q-1}^{-i}(-dc) \label{eq3} \\
       & -\frac{G(\rho_{q-1}^0)}{q^3(q-1)}\sum_{c,d\in \F_q^\ast}
  \sum_{\lambda\in \F_q}
  \sum_{x\in L_0}\sum_{h\in C_0}
  \psi_{\F_{q^3}}(zxc^{-1}+h\lambda+hd(x^q-x^{q^2})). \label{eq32}
 \end{align}
 Denote the summands in \eqref{eq3} and \eqref{eq32} by $W_3$ and $W_4$,
 respectively.
 Then $W_1=W_3+W_4$. Here, $W_4$ %\eqref{eq32}
 is reformulated as
 \begin{equation}\label{eq4}
  %\eqref{eq32}=&\sum_{c,d\in \F_q^\ast}
  %\sum_{\lambda\in \F_q}
  %\sum_{x\in L_0}\sum_{h\in C_0}
  %\psi_{\F_{q^3}}(zxc^{-1}+h\lambda+hd(x^q-x^{q^2}))\\
  W_4=\frac{1}{q^3(q-1)}\sum_{c,d\in \F_q^\ast}
  \sum_{\lambda'\in \F_q}
  \sum_{x\in C_0}\sum_{h\in L_0}
  \psi_{\F_{q^3}}(zxc^{-1}+x\lambda'+xd(h^{q^2}-h^{q})). \nonumber
 \end{equation}
 Since $R$ is a system of representatives of $(\F_{q^3}^*\setminus\F_q^*)/\F_{q}^*$, we have
 \begin{align*}
    & \{\lambda'+d(h^{q^2}-h^{q}):\lambda'\in \F_q,d\in \F_q^\ast,h\in L_0\}                         \\
  = & \{d(\lambda'+h^{q^2}-h^q):\lambda'\in \F_q,d\in \F_q^\ast,h\in L_0\}=\F_{q^3}\setminus \F_{q}.
 \end{align*}
 Then,
 \begin{align}%\label{eq5}
  %\eqref{eq4}
  W_4= & \frac{1}{q^3(q-1)}
  \sum_{c\in \F_q^\ast}
  \sum_{x\in C_0}\sum_{y\in \F_{q^3}}
  \psi_{\F_{q^3}}(zxc^{-1}+xy)-\frac{1}{q^3(q-1)}\sum_{c\in \F_q^\ast}
  \sum_{x\in C_0}\sum_{y\in \F_{q}}
  \psi_{\F_{q^3}}(zxc^{-1}+xy) \nonumber          \\
  =    & -\frac{1}{q^3(q-1)}\sum_{c\in \F_q^\ast}
  \sum_{x\in C_0}\sum_{y\in \F_{q}}
  \psi_{\F_{q^3}}(xc(z+y))=
  -\frac{1}{q^3(q-1)}\begin{cases}
   q^3-q, & \mbox{ if $z\in \F_q$},               \\
   -q,    & \mbox{ if $z\not\in \F_q$}. \nonumber
  \end{cases}
 \end{align}
 To evaluate $W_3$, we use
 \[
  \sum_{i=0}^{q-2}G(\rho_{q-1}^i)\rho_{q-1}^{-i}(-dc)
  %&\sum_{x\in\F_{q}^*}\sum_{y\in C_0}\sum_{i=0}^{q-2}\psi_{\F_{q^3}}(xy)\rho_{q-1}^{i}(xy)\rho_{q-1}^{-i}(-dc)\\
  %=&\sum_{x\in\F_{q}^*}\sum_{y\in C_0}\psi_{\F_{q^3}}(xy)\sum_{i=0}^{q-2}\rho_{q-1}^{-i}(-dcx^{-1})\\
  =(q-1)\sum_{y\in C_0}\psi_{\F_{q^3}}(-dcy),
 \]
 which follows from Lemma~\ref{partialGS}. Then  %\eqref{eq3} is reformulated as
 \begin{align}
  W_3= & \frac{1}{q^3}\sum_{c,d\in \F_q^\ast}
  \sum_{\lambda\in \F_q}
  \sum_{x\in L_0}\sum_{h,y\in C_0}
  \psi_{\F_{q^3}}(zxc^{-1}+hd(\lambda+x^q-x^{q^2})-dcy)
  %&\sum_{c,d\in \F_q^\ast}
  %\sum_{\lambda\in \F_q}
  %\sum_{x\in L_0}\sum_{h,y\in C_0}
  %\psi_{\F_{q^3}}(zxc^{-1}+hd(\lambda+x^q-x^{q^2}))\psi_{\F_{q^3}}(-dcy)
  \nonumber                                   \\
  =    & \frac{1}{q^3}\sum_{c,d\in \F_q^\ast}
  \sum_{\lambda'\in \F_q}
  \sum_{h\in L_0}\sum_{x,y\in C_0}
  \psi_{\F_{q^3}}(zxc^{-1}+x\lambda'+xd(h^{q^2}-h^{q}))\psi_{\F_{q^3}}(-dcy).\nonumber%\label{eq7}
 \end{align}
 Since $R$ is a system of coset representatives for $(\F_{q^3}^\ast\setminus \F_{q}^\ast)/ \F_q^\ast$,
 we have
 \begin{align}
  W_3= & \frac{1}{q^3}
  \sum_{c,d\in \F_q^\ast}
  \sum_{z'\in R}\sum_{x,y\in C_0}
  \psi_{\F_{q^3}}(zxc^{-1}+xdz')\psi_{\F_{q^3}}(-dcy)\nonumber      \\
  =    & \frac{1}{q^3}
  \sum_{c,e\in \F_q^\ast}
  \sum_{z'\in R}\sum_{x,y\in C_0}
  \psi_{\F_{q^3}}(zxc^{-1}+xc^{-1}ez')\psi_{\F_{q^3}}(-ey)\nonumber \\
  =    & \frac{1}{q^3}
  \sum_{e\in \F_q^\ast}
  \sum_{z'\in R}\sum_{y\in C_0}\sum_{u\in \F_{q^3}^\ast}
  \psi_{\F_{q^3}}(u(z-ez'))\psi_{\F_{q^3}}(ey).\nonumber%\label{eq8}
 \end{align}
 Here, if $z\in \F_q$,
 \[
  W_3=-(q^2+q)\sum_{e\in \F_q^\ast}
  \sum_{y\in C_0}\psi_{\F_{q^3}}(ey)=q^2+q.
 \]
 If $z\not\in \F_q$,  there is a unique $e\in \F_{q}^\ast$ such that ${e}^{-1}\in z^{-1}R\cap \F_{q}^\ast$. For this $e$, we have
 \[
  W_3=q^2+q+q^3 \psi_{\F_{q^3}}(eC_0).
 \]
 Summing up, we have
 \[
  \psi_{\F_{q^3}}(zD_3)=W_2+W_3+W_4=
  \begin{cases}
   q^2+q,                    & \mbox{if $z\in \F_q$,} \\
   -1+\psi_{\F_{q^3}}(eC_0), & \mbox{if $z\in eR$.}
  \end{cases}
 \]
 This completes the proof of the lemma.
\end{proof}

\begin{lemma}\label{Re'}
 For $e\in \F_{q}^*$, let $R_e'=\{x^{-1}y: x\in C_0, y \in \F_q^*, xy \in eR\}$. Then, $R_e'=-eR$.
\end{lemma}

\begin{proof}
 It is clear that $R_e'=eR_1'$. Therefore we only need to show that $R_1'=-R$. %, where $R_1'=\{x^{-1}y: x\in C_0, y \in \F_q^*, xy \in R\}$.
 Take $z=xy\in R$ with $x\in C_0$ and $y\in \F_q^*$. There exists $(\lambda,h)\in \F_q\times L_0$ such that $$xy=\lambda+h^{q^2}-h^q.$$
 Let $\lambda_2=\frac{-y\Tr_{q^3/q}(x^{-1})}{3}$ and $h_2=\frac{y(x^{-q^2}-x^{-q})}{3}$.
 A direct computation shows that $x^{-1}y=-(\lambda_2+h_2^{q^2}-h_2^q)$, and it is clear that $\lambda_2\in\F_q$. To complete the proof, it suffices to show that $h_2\in L_0$.

 Using $\lambda=xy-(h^{q^2}-h^q)$ and $\lambda^q=\lambda$, we obtain
 $$x^qy-h+h^{q^2}=xy-h^{q^2}+h^q,$$
 which implies that $(x^q-x)y=-3h^{q^2}$, i.e., $h=\frac{y(x^{q^2}-x^q)}{-3}$.
 Note that $h_2=\frac{y(x^{q+1}-x^{q^2+1})}{3}=\frac{xy(x^{q^2}-x^q)}{-3}=xh$. Hence, $\N_{q^3/q}(h_2)=1$.
 It is clear that $\Tr_{q^3/q}(h_2)=0$. Therefore, $h_2\in L_0$.
\end{proof}

We now give the promised proof of Proposition~\ref{odd1.2}.
\begin{proof}[Proof of Proposition~\ref{odd1.2}]
 By \eqref{tech1}, %Lemmas~\ref{chvD3} and  \ref{Re'},
 we have
 \begin{align}
  \Sigma_1%&\frac{1}{3(q^3-1)}\sum_{i=1}^{q-2}\sum_{\ell=0}^{N-1}G(\chi_2^{i}\chi_1^{-\ell})G(\chi_2^{-i}\chi_1^{-\ell})\chi_1^\ell(ab)\chi_2^{i}(ab^{-1})\sum_{x\in D_1}\chi_2^{i}\chi_1^\ell(x)\nonumber\\
  =\frac{1}{3(q^3-1)}\sum_{i=1}^{q-2}\sum_{\ell=0}^{N-1}\sum_{z\in\F_{q^3}^*}G(\chi_2^{-i}\chi_1^{-\ell})\chi_2^i\chi_1^{-\ell}(z)\chi_1^\ell(ab)\chi_2^{i}(ab^{-1})\psi_{\F_{q^3}}(zD_3).
  %\nonumber%
  \label{eq:propodd1}
 \end{align}
 By Lemma~\ref{chvD3}, continuing from \eqref{eq:propodd1},
 we have
 \begin{align}
  \Sigma_1= & \frac{N-1}{3(q^3-1)}\sum_{i=1}^{q-2}\sum_{\ell=0}^{N-1}\sum_{z\in\F_q^*}G(\chi_2^{-i}\chi_1^{-\ell})\chi_2^i\chi_1^{-\ell}(z)\chi_1^\ell(ab)\chi_2^{i}(ab^{-1})
  \nonumber                                                                                                                                                                                                          \\
            & +\frac{1}{3(q^3-1)}\sum_{e\in\F_q^*}\sum_{z\in eR}\sum_{i=1}^{q-2}\sum_{\ell=0}^{N-1}(-1+\psi_{\F_{q^3}}(eC_0))G(\chi_2^{-i}\chi_1^{-\ell})\chi_2^i\chi_1^{-\ell}(z)\chi_1^\ell(ab)\chi_2^{i}(ab^{-1}). %\nonumber%
  \label{eq:propodd2}
 \end{align}
 Here, by \eqref{tech} and Lemma~\ref{Re'},  continuing from
 \eqref{eq:propodd2},
 we have
 \begin{align}
  \Sigma_1= & \frac{N-1}{3(q^3-1)}\sum_{i=1}^{q-2}\sum_{\ell=0}^{N-1}\sum_{z\in\F_{q^3}^*}\chi_2^{-i}\chi_1^{-\ell}(z)\chi_1^\ell(ab)\chi_2^{i}(ab^{-1})\psi_{\F_{q^3}}(z\F_q^*)\label{o2.1.1}                                     \\
            & -\frac{1}{3(q^3-1)}\sum_{e\in\F_q^*}\sum_{z\in \F_{q^3}^*}\sum_{i=1}^{q-2}\sum_{\ell=0}^{N-1}\chi_2^{-i}\chi_1^{-\ell}(z)\chi_1^\ell(ab)\chi_2^{i}(ab^{-1})\psi_{\F_{q^3}}(-zeR) \label{o2.1.2}                      \\
            & +\frac{1}{3(q^3-1)}\sum_{e\in\F_q^*}\sum_{z\in \F_{q^3}^*}\sum_{i=1}^{q-2}\sum_{\ell=0}^{N-1}\psi_{\F_{q^3}}(eC_0)\chi_2^{-i}\chi_1^{-\ell}(z)\chi_1^\ell(ab)\chi_2^{i}(ab^{-1})\psi_{\F_{q^3}}(-zeR).\label{o2.1.3}
 \end{align}
 We denote the summands in \eqref{o2.1.1}, \eqref{o2.1.2} and \eqref{o2.1.3} by $\Omega_1$, $\Omega_2$ and $\Omega_3$, respectively.
 Then $\Sigma_1=\Omega_1+\Omega_2+\Omega_3$.

 Since $\psi_{\F_{q^3}}(z\F_{q}^\ast)=\sum_{y\in \F_q^\ast}\psi_{\F_{q}}(\Tr_{q^3/q}(z)y)$, we have
 \begin{align*}
  \Omega_1= & \frac{N-1}{3(q^3-1)}\sum_{i=1}^{q-2}\sum_{\ell=0}^{N-1}\sum_{z\in\F_{q^3}^*}\chi_2^{-i}\chi_1^{-\ell}(z)\chi_1^\ell(ab)\chi_2^{i}(ab^{-1})\sum_{y\in\F_{q}^*}\psi_{\F_{q}}(\Tr_{q^3/q}(z)y) \\
  =         & \frac{({\color{blue}N-1})(q-1)}{3(q^3-1)}\sum_{i=1}^{q-2}\sum_{\ell=0}^{N-1}\sum_{z\in T_0}\chi_2^{-i}\chi_1^{-\ell}(z)\chi_1^\ell(ab)\chi_2^{i}(ab^{-1})                                   \\
            & -\frac{N-1}{3(q^3-1)}\sum_{i=1}^{q-2}\sum_{\ell=0}^{N-1}\sum_{z\in \F_{q^3}^*\setminus T_0}\chi_2^{-i}\chi_1^{-\ell}(z)\chi_1^\ell(ab)\chi_2^{i}(ab^{-1})                                   \\
  =         & \frac{\color{blue}{q(N-1)}}{3(q^3-1)}\sum_{i=1}^{q-2}\sum_{\ell=0}^{N-1}\sum_{z\in T_0}\chi_2^{-i}\chi_1^{-\ell}(z)\chi_1^\ell(ab)\chi_2^{i}(ab^{-1}).
 \end{align*}
 We next evaluate $\Omega_2$. Note that
 \begin{align}
  \sum_{e\in\F_q^*}\psi_{\F_{q^3}}(-zeR)= & \sum_{e\in\F_q^*}\sum_{\lambda\in\F_q}\sum_{h\in L_0}\psi_{\F_{q^3}}(-ze(\lambda+h^{q^2}-h^q))\nonumber                                 \\
  =                                       & \sum_{e\in\F_q^*}\sum_{\lambda\in\F_q}\sum_{h\in L_0}\psi_{\F_{q}}(-e\lambda\Tr_{q^3/q}(z))\psi_{\F_{q^3}}(-ze(h^{q^2}-h^q)). \nonumber %\label{chvzeR}
 \end{align}
 If $\Tr_{q^3/q}(z)\neq 0$, %continuing from \eqref{chvzeR},
 we have $\sum_{e\in\F_q^*}\psi_{\F_{q^3}}(-zeR)=0$. If $\Tr_{q^3/q}(z)=0$, %continuing from \eqref{chvzeR},
 we have
 \begin{align*}
  \sum_{e\in\F_q^*}\psi_{\F_{q^3}}(-zeR)= & \sum_{e\in\F_q^*}\sum_{\lambda\in\F_q}\sum_{h\in L_0}\psi_{\F_{q^3}}(-ze(h^{q^2}-h^q))              \\
  =                                       & \sum_{e\in\F_q^*}\sum_{\lambda'\in\F_q}\sum_{h\in C_0}\psi_{\F_{q^3}}(-eh(z^{q}-z^{q^2})+h\lambda') \\
  =                                       & \sum_{x\in\F_{q^3}^*}\sum_{\lambda'\in\F_q}\psi_{\F_{q^3}}(x(\lambda'+z^{q^2}-z^{q}))=-q,
 \end{align*}
 where the last equality follows from  $z^{q^2}-z^q\notin \F_q$.
 Therefore, we have
 \[
  \Omega_2=\frac{q}{3(q^3-1)}\sum_{z\in T_0}\sum_{i=1}^{q-2}\sum_{\ell=0}^{N-1}\chi_2^{-i}\chi_1^{-\ell}(z)\chi_1^\ell(ab)\chi_2^{i}(ab^{-1}).
 \]
 We next evaluate $\Omega_3$. Note that
 \begin{align}
  \sum_{e\in\F_q^*}\psi_{\F_{q^3}}(eC_0)\psi_{\F_{q^3}}(-zeR)= & \sum_{e\in\F_{q}^*}\sum_{y\in C_0}\sum_{\lambda\in\F_q}\sum_{h\in L_0}\psi_{\F_{q^3}}(ey-ze(\lambda+h^{q^2}-h^q))\nonumber                              \\
  =                                                            & \sum_{e\in\F_{q}^*}\sum_{y\in C_0}\sum_{\lambda\in\F_q}\sum_{h\in L_0}\psi_{\F_q}(-e\lambda\Tr_{q^3/q}(z))\psi_{\F_{q^3}}(ey-ze(h^{q^2}-h^q)).\nonumber %\label{eC0zeR}
 \end{align}
 If $\Tr_{q^3/q}(z)\neq 0$, %continuing from \eqref{eC0zeR},
 we have  $\sum_{e\in\F_q^*}\psi_{\F_{q^3}}(eC_0)\psi_{\F_{q^3}}(-zeR)=0$. If $\Tr_{q^3/q}(z)=0$,  %continuing from \eqref{eC0zeR},
 we have
 \begin{align*}
  \sum_{e\in\F_q^*}\psi_{\F_{q^3}}(eC_0)\psi_{\F_{q^3}}(-zeR)= & \sum_{e\in\F_{q}^*}\sum_{y\in C_0}\sum_{\lambda\in\F_q}\sum_{h\in L_0}\psi_{\F_{q^3}}(ey-ze(h^{q^2}-h^q))                   \\
  =                                                            & \sum_{e\in\F_q^*}\sum_{\lambda'\in\F_q}\sum_{h,y\in C_0}\psi_{\F_{q^3}}(ey-eh(z^{q}-z^{q^2})+h\lambda')                     \\
  =                                                            & \sum_{e\in\F_q^*}\sum_{\lambda'\in\F_q}\sum_{h,y\in C_0}\psi_{\F_{q^3}}(ey(1-y^{-1}h(z^{q}-z^{q^2})+y^{-1}he^{-1}\lambda')) \\
  =                                                            & \sum_{x\in\F_{q^3}^*}\sum_{\lambda\in\F_q}\sum_{h\in C_0}\psi_{\F_{q^3}}(x(1-h(z^{q}-z^{q^2})+h\lambda))                    \\
  %=&(q^3-1)\mu_{z}-(qN-\mu_z)\\
  =                                                            & q^3\mu_{z}-qN.
 \end{align*}
 Hence we have
 \[
  \Omega_3=\frac{1}{3(q^3-1)}\sum_{z\in T_0}\sum_{i=1}^{q-2}\sum_{\ell=0}^{N-1}(q^3\mu_z-qN)\chi_2^{-i}\chi_1^{-\ell}(z)\chi_1^\ell(ab)\chi_2^{i}(ab^{-1}).
 \]
 Summing up $\Omega_i$, $i=1,2,3$, above, we obtain
 \begin{align*}
  \Sigma_1%\frac{q}{3(q-1)}\sum_{z\in T_0}\sum_{i=1}^{q-2}\sum_{\ell=0}^{N-1}\chi_2^{-i}\chi_1^{-\ell}(z)\chi_1^\ell(ab)\chi_2^{i}(ab^{-1})\\
  %&+\frac{1}{3(q^3-1)}\sum_{z\in T_0}\sum_{i=1}^{q-2}\sum_{\ell=0}^{N-1}[q^3\mu_z-qN]\chi_2^{-i}\chi_1^{-\ell}(z)\chi_1^\ell(ab)\chi_2^{i}(ab^{-1})\\
  = & \frac{q^3}{3(q^3-1)}\sum_{z\in T_0}\sum_{i=1}^{q-2}\sum_{\ell=0}^{N-1}\mu_z\chi_1^\ell(abz^{-1})\chi_2^{i}(ab^{-1}z^{-1})                    \\
  = & \frac{q^3}{3(q^3-1)}\sum_{w^{Nk+(q-1)j}\in T_0}\sum_{i=0}^{q-2}\sum_{\ell=0}^{N-1}\mu_z\chi_1^\ell(w^{(q-1)(t_0-j)})\chi_2^{i}(w^{N(u_0-k)}) \\
    & \quad -\frac{q^3}{3(q^3-1)}\sum_{w^{Nk+(q-1)j}\in T_0}\sum_{\ell=0}^{N-1}\mu_z\chi_1^\ell(w^{(q-1)(t_0-j)})                                  \\
  %=&\begin{cases}
  %\frac{q^3}{3(q^3-1)}((q-1)N\mu_{z_0}-\frac{q^3}{3(q^3-1)}N\sum_{k=0}^{q-2}\mu_{w^{Nk+(q-1)t_0}}, &\text{if~}w^{(q-1)t_0}\in T_0,\\
  %0,&\text{if~}w^{(q-1)t_0}\notin T_0,\\
  %\end{cases}\\
  = & \begin{cases}
   \frac{q^3}{3}\mu_{z_0}-\frac{q^3}{3(q-1)}\sum_{k=0}^{q-2}\mu_{w^{Nk+(q-1)t_0}}, & \text{if~}w^{(q-1)t_0}\in T_0,    \\
   0,                                                                              & \text{if~}w^{(q-1)t_0}\notin T_0, \\
  \end{cases}
 \end{align*}

 Finally, we need to compute $\sum_{k=0}^{q-2}\mu_{w^{Nk+(q-1)t_0}}$ under the assumption that $w^{(q-1)t_0}\in T_0$.
 Let $x_0=w^{(q-1)t_0}$. Since $\N_{q^3/q}(\lambda+x_0^q-x_0^{q^2})\ne 0 $ for any $\lambda\in\F_{q}$, we have
 \begin{align*}
  \sum_{\theta\in\F_{q}^*}\mu_{x_0\theta}= & \sum_{\theta\in\F_{q}^*}\#\{(y,\lambda)\in C_0\times \F_q: y=\theta(x_0^q-x_0^{q^2})-\lambda\} \\
  %=&\#\{\lambda\in\F_q: \N_{q^3/q}(\lambda+\theta(x_0^q-x_0^{q^2}))=1\}\\
  =                                        & \sum_{\theta\in\F_{q}^*}\#\{\lambda\in\F_q: \N_{q^3/q}(\lambda+x_0^q-x_0^{q^2})=\theta^{-3}\}  \\
  =                                        & \#\{\lambda\in\F_q: \N_{q^3/q}(\lambda+x_0^q-x_0^{q^2})\in \F_q^\ast\}=q.
 \end{align*}
 %Since $\N_{q^3/q}(\lambda+(x_0^q-x_0^{q^2}))=\lambda^3+\lambda\Tr_{q^3/q}((x_0^q-x_0^{q^2})^{q+1})+\N_{q^3/q}(x_0^q-x_0^{q^2})$, then $\mu_{x_0\theta}\in\{0,1,2,3\}$. Note that $\F_q=\bigcup_{\theta\in\F_{q}^*}\{\lambda\in \F_q: \N_{q^3/q}(\lambda+(x_0^q-x_0^{q^2}))=\theta^{-3}\}$. Thus
 %\[
 %|\F_q|=\sum_{k=0}^{q-2}\mu_{w^{Nk+(q-1)t_0}}=q.\]
 Therefore,
 \[\Sigma_1=\begin{cases}
   \frac{q^3}{3}\mu_{z_0}-\frac{q^4}{3(q-1)}, & \text{if~}w^{(q-1)t_0}\in T_0,    \\
   0,                                         & \text{if~}w^{(q-1)t_0}\notin T_0.
  \end{cases}\]

 Similarly, by noting that
 \begin{align*}
  G(\chi_2^{-i}\chi_1^{-\ell})\sum_{x\in D_2}\chi_2^{i}\chi_1^\ell(x)= & \sum_{z\in \F_{q^3}^*}\chi_2^{-i}\chi_1^{-\ell}(z)\psi_{\F_{q^3}}(zD_2)            \\
  =                                                                    & \sum_{z\in \F_{q^3}^*}\chi_2^{-i}\chi_1^{-\ell}(z)\psi_{\F_{q^3}}(z\beta^{-1}D_3),
 \end{align*}
 %
 %
 % $\Sigma_2$, by \eqref{tech}, note that for any $1\leq i\leq q-2$ and $0\leq \ell \leq q^2+q$, we have
 %\[G(\chi_2^{-i}\chi_1^{-\ell})\sum_{x\in D_2}\chi_2^{i}\chi_1^\ell(x)=\sum_{z\in \F_{q^3}^*}\chi_2^{-i}\chi_1^{-\ell}(z)\psi_{\F_{q^3}}(zD_2)=\sum_{z\in \F_{q^3}^*}\chi_2^{-i}\chi_1^{-\ell}(z)\psi_{\F_{q^3}}(z\beta^{-1}D_3).
 %\]
 %
 %Thus,
 we have
 \begin{align*}
  \Sigma_2=
  %\frac{1}{3(q^3-1)}\sum_{i=1}^{q-2}\sum_{\ell=0}^{N-1}\sum_{z\in\F_{q^3}^*}G(\chi_2^{i}\chi_1^{-\ell})\chi_1^\ell(ab)\chi_2^{i}(ab^{-1})\chi_2^{-i}\chi_1^{-\ell}(z)\psi_{\F_{q^3}}(z\beta^{-1}D_3)\\
  %&\frac{N-1}{3(q^3-1)}\sum_{i=1}^{q-2}\sum_{\ell=0}^{N-1}\sum_{z\in\F_{q}^*}G(\chi_2^{i}\chi_1^{-\ell})\chi_1^\ell(ab)\chi_2^{i}(ab^{-1})\chi_2^{-i}\chi_1^{-\ell}(z)\\
  %&+\frac{1}{3(q^3-1)}\sum_{i=1}^{q-2}\sum_{\ell=0}^{N-1}\sum_{z\in\beta eR}(-1+\psi_{\F_{q^3}}(eC_0))G(\chi_2^{i}\chi_1^{-\ell})\chi_1^\ell(ab)\chi_2^{i}(ab^{-1})\chi_2^{-i}\chi_1^{-\ell}(z)\\
  %=&\frac{N-1}{3(q^3-1)}\sum_{i=1}^{q-2}\sum_{\ell=0}^{N-1}\sum_{z\in\F_{q^3}^*}\chi_2^{i}\chi_1^{-\ell}(z)\chi_1^\ell(ab)\chi_2^{i}(ab^{-1})\psi_{\F_{q^3}}(z\F_q^*)\\
  %&-\frac{1}{3(q^3-1)}\sum_{e\in\F_q^*}\sum_{z\in \F_{q^3}^*}\sum_{i=1}^{q-2}\sum_{\ell=0}^{N-1}\chi_2^i\chi_1^{-\ell}(z)\chi_1^\ell(ab)\chi_2^{i}(ab^{-1})\psi_{\F_{q^3}}(-z\beta eR)\\
  %&+\frac{1}{3(q^3-1)}\sum_{e\in\F_q^*}\sum_{z\in \F_{q^3}^*}\sum_{i=1}^{q-2}\sum_{\ell=0}^{N-1}\psi_{\F_{q^3}}(eC_0)\chi_2^i\chi_1^{-\ell}(z)\chi_1^\ell(ab)\chi_2^{i}(ab^{-1})\psi_{\F_{q^3}}(-z\beta eR)\\
  %=&\frac{q^3}{3(q^3-1)}\sum_{z\in T_0}\sum_{i=1}^{q-2}\sum_{\ell=0}^{N-1}\mu_z'\chi_1^\ell(abz^{-1})\chi_2^{i}(ab^{-1}z)\\
  %=&\frac{q^3}{3(q^3-1)}\sum_{w^{Nk+(q-1)j}\in T_0}\sum_{i=0}^{q-2}\sum_{\ell=0}^{N-1}\mu_z'\chi_1^\ell(w^{(q-1)(t_0-j)})\chi_2^{i}(w^{N(u_0-k)})\\
  %&-\frac{q^3}{3(q^3-1)}\sum_{w^{Nk+(q-1)j}\in T_0}\sum_{\ell=0}^{N-1}\mu_z'\chi_1^\ell(w^{(q-1)(t_0-j)})\\
  \begin{cases}
   \frac{q^3}{3}\mu_{z_1}'-\frac{q^4}{3(q-1)}, & \text{if~}w^{(q-1)t_0}\in T_0,    \\
   0,                                          & \text{if~}w^{(q-1)t_0}\notin T_0. \\
  \end{cases}
 \end{align*}

 Finally, we evaluate $\Sigma_3$. Recall that $$\Sigma_3=-\frac{1}{3(q^3-1)}\sum_{i=1}^{q-2}\sum_{\ell=0}^{N-1}G(\chi_2^{i}\chi_1^{-\ell})G(\chi_2^{-i}\chi_1^{-\ell})\chi_1^\ell(ab)\chi_2^{i}(ab^{-1})\sum_{z\in T_0}\chi_2^{i}\chi_1^\ell(z). $$
 Since for $i\not=0$
 \begin{align*}
  \sum_{z\in T_0}\chi_2^{i}\chi_1^\ell(z)= & \sum_{x\in L_0}\sum_{y\in \F_q^*}\chi_2^i(xy)\chi_1^\ell(xy)                            \\
  =                                        & \left(\sum_{x\in L_0}\chi_1^\ell(x)\right)\left(\sum_{y\in \F_q^*}\chi_2^i(y)\right)=0,
 \end{align*}
 we have %$$\sum_{y\in\F_q^*}\chi_2^i(y)=\begin{cases}
 %q-1,&\text{if~}i=0,\\
 %0,&\text{otherwise}.\end{cases}$$
 %Therefore
 $\Sigma_3=0$. This completes the proof of the proposition. %This completes the proof of the proposition.
\end{proof}

\end{document}